\documentclass[final,3p,times]{elsarticle}

\usepackage{amssymb}
\usepackage{amsmath}
\usepackage{amsfonts}
\usepackage{amsthm}
\usepackage[english]{babel}
\usepackage{color}
\usepackage{lmodern}
\usepackage{pstricks-add}
\usepackage{dsfont}
\usepackage[cal=boondoxo]{mathalfa}
\usepackage{hyperref}
\hypersetup{hidelinks}
\usepackage{mathrsfs}

\theoremstyle{plain}
\newtheorem{theorem}{Theorem}[section]
\newtheorem{lemma}[theorem]{Lemma}

\newtheorem{proposition}[theorem]{Proposition}

\theoremstyle{definition}
\newtheorem{definition}[theorem]{Definition}

\theoremstyle{remark}
\newtheorem{remark}{Remark}

    \DeclareMathOperator\supp{supp}

    \DeclareMathOperator\Fuj{Fuj}
    
    \DeclareMathOperator\Span{span}
    \DeclareMathOperator\sign{sign}
    \DeclareMathOperator\loc{loc}
    \DeclareMathOperator\hor{H}

    \DeclareMathOperator\GN{GN}
    \DeclareMathOperator\diver{div}
    
\begin{document}

\begin{frontmatter}



\title{Critical exponent of Fujita-type for the semilinear damped wave equation on the Heisenberg group with power nonlinearity}


\author[Pisa,Waseda,Sofia]{Vladimir Georgiev}

\ead{vladimir.simeonov.gueorguiev@unipi.it}

\author[Pisa]{Alessandro Palmieri}

\ead{alessandro.palmieri.math@gmail.com}

\address[Pisa]{Department of Mathematics, University of Pisa, Largo B. Pontecorvo 5, 56127 Pisa, Italy}

\address[Waseda]{Faculty of Science and Engineering, Waseda University 3-4-1, Okubo, Shinjuku-ku, Tokyo 169-8555, Japan}

\address[Sofia]{Institute of Mathematics and Informatics-BAS Acad. G. Bonchev Str., Block 8, 1113 Sofia, Bulgaria}


\begin{abstract}
In this paper, we consider the Cauchy problem for the semilinear damped wave equation  on the Heisenberg group with power non-linearity. We prove that the critical exponent is the Fujita  exponent $p_{\Fuj}(\mathcal{Q}) = 1+2/\mathcal{Q}$, where $\mathcal{Q}$ is the homogeneous dimension of the Heisenberg group.

 On the one hand, we will prove the global existence of small data solutions for $p>p_{\Fuj}(\mathcal{Q})$ in an exponential weighted energy space. On the other hand, a blow-up result for $1<p\leq p_{\Fuj}(\mathcal{Q})$ under certain integral sign assumptions for the Cauchy data by using the test function method. 
\end{abstract}


\end{frontmatter}


\section{Introduction}\label{Intro}

In this paper we study the global in time existence of small data solutions and the blow-up in finite time of solutions to the Cauchy problem 
\begin{align}\label{CP semilinear 1}
\begin{cases}
u_{tt}-\Delta_{\hor} u+ u_t=|u|^p, & t>0, \,\, \eta \in \mathbf{H}_n, \\ u(0,\eta)=u_0(\eta), & \eta\in\mathbf{H}_n, \\ u_t(0,\eta)=u_1(\eta), & \eta\in\mathbf{H}_n,
\end{cases}
\end{align} where $p>1$ and $\Delta_{\hor}$ denotes the sub-Laplacian on $\mathbf{H}_n$ (see Subsection \ref{Subsection Heisenberg} for the definition of $\Delta_{\hor}$ and for a short summary on the Heisenberg group).

In the Euclidean case, the critical exponent of Cauchy problem for  the semilinear damped wave equations 
\begin{align}\label{CP semilinear Eucl}
\begin{cases}
u_{tt}-\Delta u+ u_t=|u|^p, & t>0, \,\, x \in \mathbb{R}^n, \\ u(0,x)=u_0(x), & x \in \mathbb{R}^n, \\ u_t(0,x)=u_1(x), & x \in \mathbb{R}^n,
\end{cases}
\end{align} is the same as for the semilinear heat equations, that is, the so-called Fujita exponent 
\begin{align*}
p_{\Fuj}(n) \doteq 1+\frac{2}{n}.
\end{align*} This fact has been proved by Todorova-Yordanov \cite{TY01} for compactly supported data and by Ikehata-Tanizawa \cite{IT05} in the not-compact case. In both works, the global existence result of small data solutions in the super-Fujita case is demonstrated in an exponentially weighted energy space. The crucial difference consists in the choice of the exponent function for the exponential weight. Furthermore, a fundamental tools in both these works are the decay estimates on $L^2(\mathbb{R}^n)$ - basis for the corresponding linear homogeneous Cauchy problem, that have been derived by Matsumura in the pioneering paper \cite{Mat76}, by using phase space analysis.

This approach with exponential weighted energy spaces has been applied also the case of time-dependent coefficients: see \cite{DLR13} for the semilinear wave equation with effective damping and \cite{Dab13,Pal17} for the scale-invariant case, respectively.

Recently, it has been shown that the critical exponent for the semilinear heat equation on the Heisenberg group is the Fujita exponent $p_{\Fuj}(\mathcal{Q})$, where $\mathcal{Q}$ is the homogeneous dimension of $\mathbf{H}_n$, and on more general stratified Lie groups (cf. \cite{Ruz18,GP19,GP19Car}).

In this paper, we will show that $p_{\Fuj}(\mathcal{Q})$ is the critical exponent for the Cauchy problem \eqref{CP semilinear 1}  as well. Concerning the existence of small data solutions which are globally defined in time for $p>p_{\Fuj}(\mathcal{Q})$, we will adapt in a suitable way the approach of \cite{TY01,IT05} with exponential weights. In fact, the counterpart of Matsumura-type estimates for the Heisenberg group is considered in \cite{Pal19}, where the group Fourier transform is employed in order to show decay estimates on $L^2(\mathbf{H}_n)$ - basis for the corresponding homogeneous linear Cauchy problem (cf. Proposition \ref{Prop Lin Estim}). On the other hand, the non-existence of global solutions when $1< p\leq p_{\Fuj}(\mathcal{Q})$, under certain integral sign  assumptions for the Cauchy data and regardless of the smallness of these, is obtained by using the so-called test function method (see \cite{MP01} or, for example, \cite{NPR16,CP19,GP19,GP19Car}).

Finally, we point out that in \cite{RT18} a global existence result for small data solutions is proved in the more general frame of graded Lie groups for the semilinear damped wave equation with an additional mass term. For that model, no further lower bound for the exponent of the nonlinearity $p>1$ has to be required, due to the exponential decay rate in $L^2$ - $L^2$ estimates for the corresponding linear homogeneous Cauchy problem (nevertheless, an upper bound for $p$ is required, although it is a technical assumption due to the application of an inequality of Gagliardo-Nirenberg type). We refer to \cite{Pal19} for further details on the differences that are produced by the absence of the mass term in the treatment of the corresponding linear problems.

\subsection{The Heisenberg group} \label{Subsection Heisenberg} 

The \emph{Heisenberg group} is the Lie group $\mathbf{H}_n=\mathbb{R}^{2n+1}$ equipped with the multiplication rule
\begin{equation*}
(x,y,\tau)\circ (x',y',\tau')= \left(x+x',y+y',\tau+\tau'+\tfrac{1}{2} (x\cdot y'-x'\cdot y)\right),
\end{equation*} where $\cdot$ denotes the standard scalar product in $\mathbb{R}^n$. 
 A system of left-invariant vector fields that span the Lie algebra $\mathfrak{h}_n$ is given by
\begin{align*}
 \ X_j  \doteq  \partial_{x_j}-\frac{y_j}{2}  \,\partial_\tau,   \ Y_j  \doteq  \partial_{y_j}+\frac{x_j}{2}  \, \partial_\tau ,  \ \partial_\tau,
\end{align*} where $1\leq j\leq n$. This system satisfies the commutation relations
\begin{align*}
[X_j,Y_k]=\delta_{jk} \, \partial_\tau \quad \mbox{for}  \ 1\leq j,k \leq n.
\end{align*} Therefore, $\mathfrak{h}_n$ admits the stratification $\mathfrak{h}_n= V_1\oplus V_2$, where $V_1\doteq \Span\{X_j,Y_j\}_{1\leq j\leq n}$ and $V_2\doteq \Span\{\partial_\tau\}$. Hence, $\mathbf{H}_n$ is a 2 step stratified Lie group, whose homogeneous dimension is $\mathcal{Q}=2n+2$.
The \emph{sub-Laplacian} on $\mathbf{H}_n$ is defined as
\begin{align}
\Delta_{\hor} & \doteq \sum_{j=1}^n (X^2_j+Y_j^2)
= \sum_{j=1}^n \big(\partial_{x_j}^2+\partial_{y_j}^2\big)+\frac{1}{4} \sum_{j=1}^n \big(x_j^2+y_j^2\big)\partial_{\tau}^2   +\sum_{j=1}^n\Big(x_j \, \partial_{y_j\tau}^2-y_j \,\partial_{ x_j \tau}^2\Big).\label{def sub laplacian}
\end{align} 
For a function $v:\mathbf{H}_n\to \mathbb{R}$, the \emph{horizontal gradient} of $v$ is 
\begin{align*}
\nabla_{\hor} v \doteq (X_1 v, \cdots, X_n v, Y_1 v, \cdots, Y_n v) \equiv \sum_{j=1}^n ((X_j v)X_j +(Y_j v)Y_j),
\end{align*} where each fiber of the horizontal subbundle $\mathrm{H}\mathbf{H}_n = \sqcup_{\eta\in \mathbf{H}_n}\mathrm{H}_\eta \mathbf{H}_n$ can be endowed with a scalar product $\langle \cdot, \cdot\rangle_{\eta}$  in such a way that  $X_1(\eta),\cdots,X_n(\eta),Y_1(\eta),\cdots,Y_n(\eta)$ are orthonormal in $(\mathrm{H}_\eta \mathbf{H}_n,\langle \cdot, \cdot\rangle_{\eta})$ for any $\eta\in \mathbf{H}_n$.
Therefore, if $X=\sum_{j=1}^n \left( \alpha_j X_j +\beta_j Y_1 \right)+\gamma \partial_\tau$ is a vector field on $\mathbf{H}_n$ with $\alpha_j,\beta_j , \gamma \in \mathcal{C}^1(\mathbf{H}_n)$ for any $j=1,\cdots, n$, the \emph{divergence} of $X$ is the function
\begin{align*}
\diver X \doteq \sum_{j=1}^n \left(X_j  \alpha_j +Y_j \beta_j\right) +\partial_\tau \gamma.
\end{align*}In particular, the sub-Laplacian may be expressed  also as $\Delta_{\hor} v = \diver (\nabla_{\hor} v)$.
 For a function $v\in L^2(\mathbf{H}_n)$ we say that $X_j v, Y_j v\in L^1_{\mathrm{loc}}(\mathbf{H}_n)$ exist \emph{in the sense of distributions}, if the integral relations
\begin{align*}
\int_{\mathbf{H_n}} \big(X_j v \big)(\eta) \, \phi(\eta) \, \mathrm{d}\eta = \int_{\mathbf{H_n}}  v (\eta) \, \big(X_j^* \phi\big)(\eta) \, \mathrm{d}\eta \quad \mbox{and} \quad \int_{\mathbf{H_n}} \big(Y_j v \big)(\eta) \, \phi(\eta) \, \mathrm{d}\eta = \int_{\mathbf{H_n}}  v (\eta) \, \big(Y_j^* \phi\big)(\eta) \, \mathrm{d}\eta
\end{align*} are fulfilled for any $\phi \in \mathcal{C}^\infty_0(\mathbf{H}_n)$, where $X_j^*=-X_j$ and $Y_j^*=-Y_j$ denote the formal adjoint operators of $X_j$ and $Y_j$, respectively. Therefore, in our framework, the \emph{Sobolev space} $H^1(\mathbf{H}_n)$ is the set of all functions $v\in L^2(\mathbf{H}_n)$ such that $X_j v, Y_j v$ exist in the sense of distributions and $X_j v, Y_j v \in  L^2(\mathbf{H}_n) $ for any $j=1,\cdots,n$, equipped with the norm
\begin{align*}
\| v\|_{H^1(\mathbf{H}_n)}^2 & \doteq \| v\|_{L^2(\mathbf{H}_n)}^2 +\| \nabla_{\hor} v\|_{L^2(\mathbf{H}_n)}^2 \\ &= \| v\|_{L^2(\mathbf{H}_n)}^2 + \sum_{j=1}^n \left( \|  X_j v\|_{L^2(\mathbf{H}_n)}^2 +\|  Y_j v\|_{L^2(\mathbf{H}_n)}^2 \right).
\end{align*}

\subsection{Notations}

In this paper, we write $f \lesssim g$, when there exists a constant $C>0$ such that $f \leq Cg$. We write $f \approx g$ when $g \lesssim  f \lesssim g$.
Throughout the article we will denote by $\psi$ the function
\begin{equation}\label{definition of psi}
\psi(t,\eta ) \doteq \frac{|x|^2+|y|^2+4|\tau|}{8(1+t)}
\end{equation} for any $\eta=(x,y,\tau)\in \mathbf{H}_n$.
Let $\sigma>0$ and $t\geq 0$. Similarly to the Euclidean case considered in \cite{TY01} and \cite{IT05}, we define the \emph{Sobolev spaces $L^2$ and $H^1$ with exponential weight $\mathrm{e}^{\sigma\psi(t,\cdot)}$ }
\begin{align*}
L^2_{\sigma\psi(t,\cdot)}(\mathbf{H}_n)& \doteq \{v\in L^2(\mathbf{H}_n):\|\mathrm{e}^{\sigma\psi(t,\cdot)}v\|_{L^2(\mathbf{H}_n)}<\infty\},\\
H^1_{\sigma\psi(t,\cdot)}(\mathbf{H}_n)& \doteq \{v\in H^1(\mathbf{H}_n):\|\mathrm{e}^{\sigma\psi(t,\cdot)}v\|_{L^2(\mathbf{H}_n)}+\|\mathrm{e}^{\sigma\psi(t,\cdot)}\nabla_{\hor} v\|_{L^2(\mathbf{H}_n)}<\infty\},
\end{align*} endowed with the norms 
 \begin{align*}
\|v\|_{L^2_{\sigma\psi(t,\cdot)}(\mathbf{H}_n)}& \doteq\|\mathrm{e}^{\sigma\psi(t,\cdot)}v\|_{L^2(\mathbf{H}_n)},\\
\|v\|_{H^1_{\sigma\psi(t,\cdot)}(\mathbf{H}_n)}&\doteq\|\mathrm{e}^{\sigma\psi(t,\cdot)} v\|_{L^2(\mathbf{H}_n)}+\| \mathrm{e}^{\sigma\psi(t,\cdot)}\nabla_{\hor} v\|_{L^2(\mathbf{H}_n)}.
\end{align*} In the local and global existence results for \eqref{CP semilinear 1} we will consider always the special case $\sigma=1$ for the function spaces to which solutions belong. Nonetheless, in order to deal with the estimates of the nonlinearity,  it is necessary sometimes to consider the general case $\sigma>0$.
Finally, we denote by $\mathcal{A}$ the space
\begin{align}\label{space for initial data}
\mathcal{A}(\mathbf{H}_n)\doteq H^1_{\psi(0,\cdot)}(\mathbf{H}_n)\times L^2_{\psi(0,\cdot)}(\mathbf{H}_n) 
\end{align} to which initial data will be required to belong to.

\section{Main results}

Let us state the main theorems that will be proved in the present article.

\begin{theorem}\label{thm local existence} Let $n\geq 1$. Let us assume $p>1$ such that $p\leq p_{\GN}(\mathcal{Q}) \doteq  \frac{\mathcal{Q}}{\mathcal{Q}-2}$.
 Then for each initial data $(u_0,u_1)\in \mathcal{A} (\mathbf{H}_n)
 $ there exists a maximal existence time $T_\mathrm{max}\in(0,\infty]$ such that the Cauchy problem \eqref{CP semilinear 1} has a unique solution $u\in\mathcal{C}([0,T_\mathrm{max}),H^1(\mathbf{H}_n))\cap \mathcal{C}^1([0,T_\mathrm{max}),L^2(\mathbf{H}_n))$.
 
 Moreover, for any $T\in (0,T_\mathrm{max})$ it holds
\begin{align*}
\sup_{t\in[0,T]}\left(\|\mathrm{e}^{\psi(t,\cdot) }v(t,\cdot)\|_{L^2(\mathbf{H}_n)}+\|\mathrm{e}^{\psi(t,\cdot)}\nabla_{\hor} v(t,\cdot)\|_{L^2(\mathbf{H}_n)}+\|\mathrm{e}^{\psi(t,\cdot) }v_t(t,\cdot)\|_{L^2(\mathbf{H}_n)}\right)<\infty.
\end{align*} Finally, if $T_\mathrm{max}<\infty$, then
\begin{align*}
\limsup_{T\to T_\mathrm{max}^-}\left(\|\mathrm{e}^{\psi(t,\cdot) }v(t,\cdot)\|_{L^2(\mathbf{H}_n)}+\|\mathrm{e}^{\psi(t,\cdot)}\nabla_{\hor} v(t,\cdot)\|_{L^2(\mathbf{H}_n)}+\|\mathrm{e}^{\psi(t,\cdot) }v_t(t,\cdot)\|_{L^2(\mathbf{H}_n)}\right)=\infty.
\end{align*}
\end{theorem}

The previous local existence result is a preparatory result to the next global existence theorem, whose proof is based on a contradiction argument that requires the existence of local in time solutions for \eqref{CP semilinear 1}. 

\begin{theorem} \label{thm glob exi exp data} Let $n\geq 1$. Let us consider $1<p\leq p_{\GN}(\mathcal{Q}) $ such that 
$p>p_{\Fuj}\left(\mathcal{Q}\right)$.
 Then, there exists $\varepsilon_0>0$ such that for any initial data 
\begin{align}\label{global existence eponential weight data cond}
(u_0,u_1)\in \mathcal{A}(\mathbf{H}_n)\quad \mbox{satisfying}\quad \|(u_0,u_1)\|_{\mathcal{A}(\mathbf{H}_n)}\leq \varepsilon_0
\end{align} there is a unique solution $u\in \mathcal{C}([0,\infty),H^1_{\psi(t,\cdot)}(\mathbf{H}_n))\cap \mathcal{C}^1([0,\infty),L^2_{\psi(t,\cdot)}(\mathbf{H}_n))$ to the Cauchy problem \eqref{CP semilinear 1}.
Moreover, $u$ satisfies the following estimates
\begin{align*}
\|u(t,\cdot)\|_{L^2(\mathbf{H}_n)}&\lesssim (1+t)^{-\frac{\mathcal{Q}}{2}}\|(u_0,u_1)\|_\mathcal{A(\mathbf{H}_n)},\\
\|\nabla_{\hor} u(t,\cdot)\|_{L^2(\mathbf{H}_n)}&\lesssim (1+t)^{-\frac{\mathcal{Q}}{4}-\frac{1}{2}}\|(u_0,u_1)\|_\mathcal{A(\mathbf{H}_n)},\\
\|u_t(t,\cdot)\|_{L^2(\mathbf{H}_n)}&\lesssim  (1+t)^{-\frac{\mathcal{Q}}{4}-1}\|(u_0,u_1)\|_\mathcal{A(\mathbf{H}_n)},\\
\|\mathrm{e}^{\psi(t,\cdot)}\nabla_{\hor} u (t,\cdot)\|_{L^2(\mathbf{H}_n)}&\lesssim \|(u_0,u_1)\|_\mathcal{A(\mathbf{H}_n)}, \\
\|\mathrm{e}^{\psi(t,\cdot)} u_t(t,\cdot)\|_{L^2(\mathbf{H}_n)}&\lesssim \|(u_0,u_1)\|_\mathcal{A(\mathbf{H}_n)}
\end{align*} for any $t\geq 0$.
\end{theorem}

\begin{remark}
 Let us point out that the requirement $(u_0,u_1)\in \mathcal{A}(\mathbf{H}_n)$ in Theorem \ref{thm glob exi exp data} is stronger than the assumption $(u_0,u_1)\in (H^1(\mathbf{H}_n)\cap L^1(\mathbf{H}_n))\times (L^2(\mathbf{H}_n)\cap L^1(\mathbf{H}_n))$. Indeed,   the embedding
\begin{align*}
L^2_{\sigma\psi(t,\cdot)}(\mathbf{H}_n)\hookrightarrow L^1(\mathbf{H}_n)\cap L^2(\mathbf{H}_n)
\end{align*} holds for any $\sigma>0$ and $t\geq 0$. By using Cauchy-Schwarz inequality and the nonnegativity of $\psi$, it results
\begin{align}
\|v\|_{L^1(\mathbf{H}_n)}&\lesssim (1+t)^{\frac{\mathcal{Q}}{4}}\|\mathrm{e}^{\sigma\psi(t,\cdot)}v\|_{L^2(\mathbf{H}_n)},\label{L1-L2 weight est}\\
\|v\|_{L^2(\mathbf{H}_n)}&\lesssim \|\mathrm{e}^{\sigma\psi(t,\cdot)}v\|_{L^2(\mathbf{H}_n)}.\label{L2-L2 weight est}
\end{align} In order to  prove \eqref{L1-L2 weight est}, we employed the value of the integral of Gaussian-type
\begin{align*}
\int_{\mathbf{H}_n} \mathrm{e}^{-2\sigma \psi(t,\eta)} \, \mathrm{d}\eta & = \int_{\mathbb{R}^n} \mathrm{e}^{-\frac{\sigma|x|^2}{4(1+t)}} \, \mathrm{d}x \int_{\mathbb{R}^n} \mathrm{e}^{-\frac{\sigma|y|^2}{4(1+t)}} \, \mathrm{d}y \int_{\mathbb{R}} \mathrm{e}^{-\frac{\sigma|\tau|}{4(1+t)}} \, \mathrm{d}\tau  \\& = 2^{\mathcal{Q}+1}\pi^{\frac{\mathcal{Q}}{2}-1} \sigma^{-\frac{\mathcal{Q}}{2}} (1+t)^{\frac{\mathcal{Q}}{2}}.
\end{align*}  Furthermore, by H\"{o}lder's interpolation inequality we have also the embedding of $L^2_{\sigma\psi(t,\cdot)}(\mathbf{H}_n)$ in each $L^r(\mathbf{H}_n)$ for any $r\in [1,2]$, where the embedding constant depends on $t$, clearly.
\end{remark}

\begin{theorem}\label{thm blow up} Let $n\geq 1$. Let $u_0,u_1 \in L^1(\mathbf{H}_n)$ such that 
\begin{align}
\liminf_{R\to \infty} \int_{\mathcal{D}_R}\big(u_0(\eta)+u_1(\eta)\big)  \mathrm{d}\eta >0, \label{assumption intial data TFM}
\end{align} where $\mathcal{D}_R\doteq B_n(R)\times B_n(R)\times [-R^2,R^2]$. Let us assume that $u\in L^p_{\mathrm{loc}}([0,T)\times \mathbb{R}^n)$ is a solution to \eqref{CP semilinear 1}, with life-span $T>0$.  If $1<p\leq p_{\Fuj}\left(Q\right)$,  then $T<\infty$, that is, the solutions $u$ blows up in finite time.
\end{theorem}

The next sections are organized as follows: in Section \ref{Section Overview} we explain the strategy for the proofs of Theorems \ref{thm local existence} and \ref{thm glob exi exp data} and we derive some  important estimates by using some remarkable properties of the function $\psi$; in Sections \ref{Section GN ineq} and \ref{Section Linear estimates} we derive a weighted version of the Gagliardo-Nirenberg inequality on $\mathbf{H}_n$ and we recall $L^2(\mathbf{H}_n)$ - $L^2(\mathbf{H}_n)$ estimates (with possible additional $L^1(\mathbf{H}_n)$ regularity) for the solution of \eqref{linear CP}, respectively; then, we prove Theorems \ref{thm local existence} and \ref{thm glob exi exp data} in Section \ref{LE Section} and in Section \ref{SDGE Section}, respectively; finally, we prove the blow-up result in Section \ref{Section Blow-up}.

\section{Overview on our approach} \label{Section Overview}

We apply Duhamel's principle in order to write the solution to \eqref{CP semilinear 1}. Because the linear equation related to the semi-linear equation in \eqref{CP semilinear 1} is invariant by time translations, we need to derive decay estimates  for  linear Cauchy problem
\begin{align}\label{linear CP}
\begin{cases}
u_{tt}-\Delta_{\hor} u+ u_t=0, & t> 0, \,\, \eta\in \mathbf{H}_n, \\ u(0,\eta)=u_0(\eta), & \eta\in\mathbf{H}_n, \\ u_t(0,\eta)=u_1(\eta), & \eta\in\mathbf{H}_n.
\end{cases}
\end{align} 
Let us fix now some notations for the linear Cauchy problem \eqref{linear CP}.
We denote by $E_0(t,\eta), E_1(t,\eta)$ the fundamental solutions to the Cauchy problem \eqref{linear CP}, i.e., the distributional solutions with data $(u_0,u_1)=(\delta_0,0)$ and $(u_0,u_1)=(0,\delta_0)$, respectively, where $\delta_0$ is the Dirac distribution in the $\eta$ variable. Also, if we denote by $\ast_{(\eta)}$ the group convolution with respect to the $\eta$ variable, we may represent the solution to the Cauchy problem \eqref{linear CP} as 
\begin{align*}
u(t,\eta)=u_0(\eta)\ast_{(\eta)} E_0(t,\eta)+u_1(\eta)\ast_{(\eta)} E_1(t,\eta).
\end{align*}
According to Duhamel's principle adapted to the case of Lie groups, we get 
\begin{align}\label{Duhamel term}
u(t,\eta)=\int_0^t F(s,\eta)  \ast_{(\eta)} E_1(t-s,\eta) \, \mathrm{d} s
\end{align} as mild solution to the inhomogeneous Cauchy problem 
\begin{align*}
\begin{cases}
u_{tt}-\Delta_{\hor} u+ u_t=F(t,\eta), & t> 0, \,\, \eta\in \mathbf{H}_n, \\ u(0,\eta)=0, & \eta\in\mathbf{H}_n, \\ u_t(0,\eta)=0, & \eta\in\mathbf{H}_n.
\end{cases}
\end{align*} In particular, we used the fact that the identity $L(v\ast_{(\eta)} E_1)=v\ast_{(\eta)} L( E_1)$ holds for any left invariant differential operator $L$ on $\mathbf{H}_n$. 

Therefore, we consider as mild solutions to \eqref{CP semilinear 1} on $(0,T)\times \mathbf{H}_n$ any fixed point of the nonlinear integral operator $N$ defined as follows:
\begin{align}\label{definition N operator}
u\in X(T)\to Nu(t,\eta)\doteq u_0(\eta)\ast_{(\eta)} E_0(t,\eta) +u_1(\eta)\ast_{(\eta)} E_1(t,\eta) +\int_0^t |u(s,\eta)|^p \ast_{(\eta)} E_1(t-s,\eta) \, \mathrm{d}s
\end{align} for a suitably chosen space $X(T)$ (here $T$ denotes the lifespan of the solution).

In particular, in Theorem \ref{thm glob exi exp data} the global in time solutions we are interested in are solution in $X(T)$ to the integral equation 
 \begin{align*}
 u(t,\eta)=u_0(\eta)\ast_{(\eta)} E_0(t,\eta) +u_1(\eta)\ast_{(\eta)} E_1(t,\eta) +\int_0^t |u(s,\eta)|^p \ast_{(\eta)} E_1(t-s,\eta) \, \mathrm{d}s
 \end{align*} which can be extended for all positive times.

Also, one difficulty in the proof of the local existence result for large data and of the  global existence result for small data, respectively, consists in the choice of the space $X(T)$.
In this paper, we restrict our consideration to the weighted energy space
\begin{align*}
X(T)=\mathcal{C}([0,T],H^1_{\psi(t,\cdot)}(\mathbf{H}_n))\cap \mathcal{C}^1([0,T],L^2_{\psi(t,\cdot)}(\mathbf{H}_n)),
\end{align*} both in Theorem \ref{thm local existence} and in Theorem \ref{thm glob exi exp data}. As we will see, the crucial difference lies in the choice of the norm for $X(T)$ (cf. Section \ref{LE Section} and Section \ref{SDGE Section}).

We analyze now some properties of the function $\psi$, defined in \eqref{definition of psi}, that will be useful in the proof of our main results.
Straightforward computations lead to 
\begin{align*}
\psi_t(t,\eta) & = - \frac{|x|^2+|y|^2+4|\tau|}{8(1+t)^2}, \ \
X_j \psi(t,\eta)  = \frac{x_j-\sign(\tau)y_j}{4(1+t)}, \ \ Y_j \psi(t,\eta)  = \frac{y_j+\sign(\tau)x_j}{4(1+t)}
\end{align*} for any $j=1,\cdots,n$. Let us point out explicitly that in the following we consider weak derivatives, so, the previous expressions for $X_j \psi$ and $Y_j \psi$ are in the sense of distributions.
Consequently, the following inequalities are satisfied
\begin{align}
|\nabla_{\hor} \psi(t,\eta)|^2+\psi_t(t,\eta) & = -\frac{|\tau |}{2(1+t)^2} \leq 0, \label{prop grad psi}\\
\Delta_{\hor}\psi(t,\eta)&=\frac{n}{2(1+t)}+ \frac{|x|^2+|y|^2}{4(1+t)}\delta_0(\tau)\label{prop lapl psi}
\end{align}
for any $t\geq 0$ and any $\eta\in\mathbf{H}_n$, where $\delta_0(\tau)$ denotes the Dirac delta in 0 with respect to the $\tau$ variable. 
 We derive now some fundamental relations that will play a crucial role in the next sections.
  The first one is the identity
\begin{align}
\mathrm{e}^{2\psi}u_t\left(u_{tt}-\Delta_{\hor} u +u_t\right) & =\frac{\partial}{\partial t}\left(\frac{\mathrm{e}^{2\psi}}{2}\left(|u_t|^2+|\nabla_{\hor} u|^2\right)\right)-\diver(\mathrm{e}^{2\psi}u_t\nabla_{\hor} u)\notag +\frac{\mathrm{e}^{2\psi}}{\psi_t}u_t^2\left(|\nabla_{\hor}\psi|^2+\psi_t\right) \\ & \quad -\psi_t \mathrm{e}^{2\psi}u_t^2 -\frac{\mathrm{e}^{2\psi}}{\psi_t}|u_t\nabla_{\hor}\psi-\psi_t\nabla_{\hor} u|^2 ,\label{fundamental equality 1} 
\end{align} where $|\nabla_{\hor} v|^2 = \sum_{j=1}^n\left( |X_j v|^2+|Y_j v|^2\right)$ is the Euclidean norm of $\nabla_{\hor} v$ provided that we use the identification $\nabla_{\hor} v \simeq (X_1u, \cdots, X_n u,Y_1u,\cdots, Y_n u) : \mathbf{H}_n \to \mathbb{R}^{2n}$ or, equivalently, we consider on each fiber of the horizontal subbundle $\mathrm{H}_\eta \mathbf{H}_n$ the norm induced by the scalar product $\langle \cdot, \cdot \rangle_\eta$. Let us verify the validity of \eqref{fundamental equality 1}. Using the fact that the sub-Laplacian can be expressed as the divergence of the horizontal gradient and the identity $$\diver (\alpha X) = \alpha \diver X +  X( \alpha) $$ for any $\alpha\in \mathcal{C}^1(\mathbf{H}_n)$ and any horizontal vector field $X$ on $\mathbf{H}_n$, we get
\begin{align}
\mathrm{e}^{\psi}u_t \, \Delta_{\hor} u &=\mathrm{e}^{\psi}u_t \diver (\nabla_{\hor} u)=\diver(\mathrm{e}^{2\psi}u_t\nabla_{\hor} u)-(\nabla_{\hor} u)(\mathrm{e}^{2\psi}u_t) \notag \\
& =\diver(\mathrm{e}^{2\psi}u_t\nabla_{\hor} u)-\Big(\sum_{j=1}^n X_j(u) X_j+Y_j(u) Y_j\Big)(\mathrm{e}^{2\psi}u_t) \notag\\
& =\diver(\mathrm{e}^{2\psi}u_t\nabla_{\hor} u)-\sum_{j=1}^n X_j(u)\big(\mathrm{e}^{2\psi} X_j(u_t)+2\mathrm{e}^{2\psi}u_tX_j(\psi)\big)- \sum_{j=1}^n Y_j(u)\big(\mathrm{e}^{2\psi} Y_j(u_t)+2\mathrm{e}^{2\psi}u_tY_j(\psi)\big). \label{diver grad hor}
\end{align}
Since 
\begin{align*}
\sum_{j=1}^n X_j(u) X_j(u_t) \,  \mathrm{e}^{2\psi}  & = \sum_{j=1}^n  \frac{\mathrm{e}^{2\psi}}{2} \,  \frac{\partial}{\partial t} |X_j(u)|^2   \\ & = \sum_{j=1}^n  \bigg( \frac{\partial}{\partial t} \bigg( \frac{ \mathrm{e}^{2\psi}}{2} \, |X_j(u)|^2  \bigg) - \psi_t\mathrm{e}^{2\psi} |X_j(u)|^2  \bigg)
\end{align*} and, analogously, 
\begin{align*}
\sum_{j=1}^n Y_j(u) Y_j(u_t) \,  \mathrm{e}^{2\psi}  &=\sum_{j=1}^n  \bigg( \frac{\partial}{\partial t} \bigg( \frac{ \mathrm{e}^{2\psi}}{2} \, |Y_j(u)|^2  \bigg) - \psi_t\mathrm{e}^{2\psi} |Y_j(u)|^2  \bigg),
\end{align*} it follows
\begin{align} \label{piece 1}
\sum_{j=1}^n \big( X_j(u) X_j(u_t) \,  \mathrm{e}^{2\psi}+  Y_j(u) Y_j(u_t) \,  \mathrm{e}^{2\psi} \big) =  \frac{\partial}{\partial t} \bigg( \frac{ \mathrm{e}^{2\psi}}{2} \, |\nabla_{\hor} u|^2  \bigg) - \psi_t\mathrm{e}^{2\psi} |\nabla_{\hor} u|^2. 
\end{align} On the other hand, using the (0,2) symmetric tensor $\langle \cdot,\cdot\rangle$ on $\mathbf{H}_n$, whose restriction to each fiber $\mathrm{H}_\eta \mathbf{H}_n$ of the horizontal subbundle is the scalar product $\langle \cdot,\cdot\rangle_{\eta}$ with orthonormal basis given by the canonical generators of the horizontal layer (cf. Subsection \ref{Subsection Heisenberg}), we have
\begin{align*}
\left\langle u_t \nabla_{\hor} \psi , \psi_t  \nabla_{\hor} u \right\rangle = \sum_{j=1}^n \big(u_t \psi_t X_j(\psi) X_j(u)+ u_t \psi_t Y_j(\psi) Y_j(u)\big). 
\end{align*} Consequently,
\begin{align}
2\sum_{j=1}^n \mathrm{e}^{2\psi} u_t  X_j(\psi) X_j(u)+ \mathrm{e}^{2\psi} u_t  Y_j(\psi) Y_j(u) & = 2 \, \frac{\mathrm{e}^{2\psi}}{\psi_t} \left\langle u_t \nabla_{\hor} \psi , \psi_t  \nabla_{\hor} u \right\rangle \notag \\
& =  \frac{\mathrm{e}^{2\psi}}{\psi_t} \Big(  u_t^2 |\nabla_{\hor} \psi|^2 + \psi_t^2  |\nabla_{\hor} u |^2 -|u_t\nabla_{\hor} \psi-\psi_t\nabla_{\hor} u |^2\Big). \label{piece 2}
\end{align}
Combining \eqref{diver grad hor}, \eqref{piece 1} and \eqref{piece 2}, we get 
\begin{align}
\mathrm{e}^{\psi}u_t \, \Delta_{\hor} u &= \diver(\mathrm{e}^{2\psi}u_t\nabla_{\hor} u) - \frac{\partial}{\partial t} \bigg( \frac{ \mathrm{e}^{2\psi}}{2} \, |\nabla_{\hor} u|^2  \bigg)  + \frac{\mathrm{e}^{2\psi}}{\psi_t} \Big( |u_t\nabla_{\hor} \psi-\psi_t\nabla_{\hor} u |^2 - u_t^2 |\nabla_{\hor} \psi|^2 \Big). \label{piece 3}
\end{align}
Furthermore,
\begin{align}
\mathrm{e}^{\psi}u_t u_{tt}&=\frac{\mathrm{e}^{2\psi}}{2}\partial_t |u_t|^2=\frac{\partial}{\partial t}\left(\frac{\mathrm{e}^{2\psi}}{2}|u_t|^2\right)-\psi_t \mathrm{e}^{2\psi}u_t^2. \label{piece 4}
\end{align} By \eqref{piece 3} and \eqref{piece 4} we find immediately \eqref{fundamental equality 1}.

The second fundamental relation is the upcoming inequality, that is  obtained by plugging the nonlinear term on the left hand side of \eqref{fundamental equality 1}. If $u$ is a solution of the equation \eqref{CP semilinear 1}, since $$\mathrm{e}^{2\psi}u_t|u|^p=\mathrm{e}^{2\psi}\partial_t\left(\frac{|u|^pu}{p+1}\right)=\frac{\partial}{\partial t}\left(\mathrm{e}^{2\psi}\frac{|u|^pu}{p+1}\right)-2\psi_t \mathrm{e}^{2\psi}\frac{|u|^pu}{p+1},$$ then, from \eqref{prop grad psi} we get immediately
\begin{align}
&\frac{\partial}{\partial t}\left(\frac{\mathrm{e}^{2\psi}}{2}\left(|u_t|^2+|\nabla_{\hor} u|^2\right)-\mathrm{e}^{2\psi}\frac{|u|^pu}{p+1}\right)\notag\\&\quad =\diver(\mathrm{e}^{2\psi}u_t\nabla_{\hor} u)- \frac{\mathrm{e}^{2\psi}}{\psi_t}u_t^2\left(|\nabla_{\hor}\psi|^2+\psi_t\right)+\psi_t \mathrm{e}^{2\psi} u_t^2+\frac{\mathrm{e}^{2\psi}}{\psi_t}|u_t\nabla_{\hor}\psi-\psi_t\nabla_{\hor} u|^2  -2\psi_t \mathrm{e}^{2\psi}\frac{|u|^pu}{p+1}\notag\\ &\quad \leq \diver(\mathrm{e}^{2\psi}u_t\nabla_{\hor} u)-2\psi_t \mathrm{e}^{2\psi}\frac{|u|^pu}{p+1},\label{fundamental equality 2}
\end{align} where in the last inequality we used \eqref{prop grad psi} and $\psi_t\leq 0$.
We stress that  in Sections \ref{LE Section} and \ref{SDGE Section} an important role in the derivation of weighted energy estimates will be played by \eqref{fundamental equality 1} and \eqref{fundamental equality 2}.

\section{Gagliardo-Nirenberg type inequalities} \label{Section GN ineq}

In the proof of Theorems \ref{thm local existence} and \ref{thm glob exi exp data}, we make use of the following inequalities of Gagliardo-Nirenberg type. We begin with the Gagliardo-Nirenberg inequality in $\mathbf{H}_n$ (cf. \cite{CM13,RT18}).

\begin{lemma} \label{Lemma dis GN Hn} Let $n\geq 1$. Let us consider $2\leq q \leq 2+\frac{2}{n}= \frac{2\mathcal{Q}}{\mathcal{Q}-2}$. Then, the following Gagliardo-Nirenberg inequality holds
\begin{align*}
\| v\|_{L^q(\mathbf{H}_n)} \leq C \,  \| \nabla_{\hor} v\|_{L^2(\mathbf{H}_n)}^{\theta(q)}\| v\|_{L^2(\mathbf{H}_n)}^{1-\theta(q)}
\end{align*} for any $v\in H^1(\mathbf{H}_n)$, where  where $C$ is a nonnegative constant and $\theta(q)\in [0,1]$ is defined by
\begin{align}\label{def theta(q)}
\theta(q)\doteq \mathcal{Q} \left(\tfrac{1}{2}-\tfrac{1}{q}\right).
\end{align}
\end{lemma}

\begin{lemma}\label{1 lemma GN with weight} Let $\sigma>0$, $t\geq 0$. Then, the following estimate 
\begin{align*}
2^{-1} \sigma n \, (1+t)^{-1}\|\mathrm{e}^{\sigma\psi(t,\cdot)}v\|^2_{L^2(\mathbf{H}_n)}+\|\nabla_{\hor}(\mathrm{e}^{\sigma\psi(t,\cdot)}v)\|^2_{L^2(\mathbf{H}_n)}\leq \|\mathrm{e}^{\sigma\psi(t,\cdot)}\nabla_{\hor} v\|^2_{L^2(\mathbf{H}_n)}
\end{align*} holds for any $v\in H^1_{\sigma\psi(t,\cdot)}(\mathbf{H}_n)$.
\end{lemma}

\begin{proof}
Let us set $f=\mathrm{e}^{\sigma\psi}v$. Then, straightforward computations lead to
\begin{align*}
 \mathrm{e}^{\sigma\psi}\nabla_{\hor} v=\nabla_{\hor} f-\sigma f\nabla_{\hor} \psi.
\end{align*}
 Hence, 
\begin{align}
\|\mathrm{e}^{\sigma\psi(t,\cdot)}\nabla_{\hor} v\|^2_{L^2(\mathbf{H}_n)} & = \|\nabla_{\hor} f(t,\cdot)\|^2_{L^2(\mathbf{H}_n)}+\sigma^2 \|(f\nabla_{\hor} \psi)(t,\cdot)\|^2_{L^2(\mathbf{H}_n)}-2\sigma\big(\nabla_{\hor} f(t,\cdot), (f\nabla_{\hor}\psi)(t,\cdot)\big)_{L^2(\mathbf{H}_n)} \notag \\
& \geq  \|\nabla_{\hor} f(t,\cdot)\|^2_{L^2(\mathbf{H}_n)}-2\sigma\big(\nabla_{\hor} f(t,\cdot), (f\nabla_{\hor}\psi)(t,\cdot)\big)_{L^2(\mathbf{H}_n)}. \label{lower bound e sigma psi Nabla hor v}
\end{align}
Integrating by parts, we have

Note that we may integrate by parts 
\begin{align*}
\int_{\mathbf{H}_n} X_j g(\eta) \cdot h (\eta) \, \mathrm{d}\eta = - \int_{\mathbf{H}_n} g(\eta) \cdot  X_j h (\eta) \, \mathrm{d}\eta, \quad \int_{\mathbf{H}_n} Y_j g(\eta) \cdot h (\eta) \, \mathrm{d}\eta = - \int_{\mathbf{H}_n} g(\eta) \cdot  Y_j h (\eta) \, \mathrm{d}\eta
\end{align*} for any $g,h\in \mathcal{C}^1_0(\mathbf{H}_n)$ and for any $j=1,\cdots,n$. So, using a partition of the unity we may remove the compact support assumption while a density argument provides the result for weak derivatives.

\begin{align}
\big(\nabla_{\hor} f(t,\cdot), & (f \, \nabla_{\hor}\psi)(t,\cdot)\big)_{L^2(\mathbf{H}_n)} \notag \\ &= \sum_{j=1}^n \int_{\mathbf{H}_n} \big(f X_j f  X_j\psi + f Y_j f  \, Y_j\psi\big)(t,\eta) \, \mathrm{d} \eta = \frac{1}{2} \sum_{j=1}^n \int_{\mathbf{H}_n} \big( X_j |f|^2  X_j\psi +  Y_j |f|^2  \, Y_j\psi\big)(t,\eta) \, \mathrm{d} \eta\notag \\ 
& =- \frac{1}{2}\sum_{j=1}^n \int_{\mathbf{H}_n} \big( |f|^2  X^2_j \psi +  |f|^2  \, Y_j^2\psi\big)(t,\eta) \, \mathrm{d} \eta  =-\frac{1}{2}\int_{\mathbf{H}_n}\big(|f|^2 \Delta_{\hor}\psi\big)  (t,\eta) \, \mathrm{d}\eta \notag \\ &\leq -\frac{n}{4(1+t)}\|f(t,\cdot)\|^2_{L^2(\mathbf{H}_n)}, \label{L2 norm scalar prod}
\end{align} where in the last step we used \eqref{prop lapl psi}. Note that we may consider the trace of the function $|f|^2$ on the hypersurface with equation $\tau=0$, since the existence of trace operators is known in the literature for the Heisenberg group (cf. \cite{Pes94,BP99,BCX05,BCX09}).
Consequently, combining \eqref{lower bound e sigma psi Nabla hor v} and \eqref{L2 norm scalar prod}, we get the desired  estimate.
\end{proof}

\begin{lemma}\label{2 lemma GN with weight}  Let $n\geq 1$, $\sigma\in (0,1]$ and  $t\geq 0$. Let us consider $2\leq q \leq 2+\frac{2}{n}= \frac{2\mathcal{Q}}{\mathcal{Q}-2}$.  Then, the following weighted Gagliardo-Nirenberg inequality 
\begin{align}\label{GN inequality weighted}
\|\mathrm{e}^{\sigma\psi(t,\cdot)}v\|_{L^q(\mathbf{H}_n)}\leq C (1+t)^{(1-\theta(q))/2}\|\nabla_{\hor} v\|_{L^2(\mathbf{H}_n)}^{1-\sigma}\|\mathrm{e}^{\psi(t,\cdot)}\nabla_{\hor} v\|_{L^2(\mathbf{H}_n)}^{\sigma},
\end{align}  holds for any $v\in H^1_{\psi(t,\cdot)}(\mathbf{H}_n)$, where $C$ is a nonnegative constant and $\theta(q)$ is defined by \eqref{def theta(q)}.
\end{lemma}

\begin{proof}
Let us prove first that  $v\in H^1_{\psi(t,\cdot)}(\mathbf{H}_n)$ implies $v\in H^1_{\sigma\psi(t,\cdot)}(\mathbf{H}_n)$ for any $\sigma\in (0,1]$.
 By H\"{o}lder's inequality we find
\begin{align}
\|\mathrm{e}^{\sigma\psi(t,\cdot)}\nabla_{\hor} v\|_{L^2(\mathbf{H}_n)}^2&=\int_{\mathbf{H}_n}\mathrm{e}^{2\sigma\psi(t,\eta)}|\nabla_{\hor} v(\eta)|^{2\sigma}|\nabla_{\hor} v(\eta)|^{2(1-\sigma)}\mathrm{d}\eta \notag \\&\leq \|\mathrm{e}^{2\sigma\psi(t,\cdot)}|\nabla_{\hor} v|^{2\sigma}\|_{L^\frac{1}{\sigma}(\mathbf{H}_n)} \||\nabla_{\hor} v|^{2(1-\sigma)}\|_{L^\frac{1}{1-\sigma}(\mathbf{H}_n)}\notag \\ &=\|\mathrm{e}^{\psi(t,\cdot)}\nabla_{\hor} v\|^{2\sigma}_{L^2(\mathbf{H}_n)}\|\nabla_{\hor} v\|^{2(1-\sigma)}_{L^2(\mathbf{H}_n)}.\label{2 lemma GN with weight 1 est}
\end{align}
 In a similar way, it results
\begin{align*}
\|\mathrm{e}^{\sigma\psi(t,\cdot)}v\|_{L^2(\mathbf{H}_n)}^2&=\int_{\mathbf{H}_n}\mathrm{e}^{2\sigma\psi(t,\eta)}| v(\eta)|^{2\sigma}| v(\eta)|^{2(1-\sigma)}\mathrm{d}\eta  \\&\leq \|\mathrm{e}^{2\sigma\psi(t,\cdot)}| v|^{2\sigma}\|_{L^\frac{1}{\sigma}(\mathbf{H}_n)} \|| v|^{2(1-\sigma)}\|_{L^\frac{1}{1-\sigma}(\mathbf{H}_n)} \\ & =\|\mathrm{e}^{\psi(t,\cdot)} v\|^{2\sigma}_{L^2(\mathbf{H}_n)}\| v\|^{2(1-\sigma)}_{L^2(\mathbf{H}_n)}.
\end{align*}
 So, we have that $f=\mathrm{e}^{\sigma\psi}v$ satisfies $f(t,\cdot)\in H^1(\mathbf{H}_n)$ and by Lemma \ref{1 lemma GN with weight}
\begin{align}
\|f(t,\cdot)\|_{L^2(\mathbf{H}_n)}&\lesssim (1+t)^{1/2}\|\mathrm{e}^{\sigma\psi(t,\cdot)}\nabla_{\hor} v\|_{L^2(\mathbf{H}_n)},\label{2 lemma GN with weight 2 est}\\ \|\nabla_{\hor} f(t,\cdot)\|_{L^2(\mathbf{H}_n)}&\leq \|\mathrm{e}^{\sigma\psi(t,\cdot)}\nabla_{\hor} v\|_{L^2(\mathbf{H}_n)} \label{2 lemma GN with weight 3 est}
\end{align}
for any $t\geq 0$. 
Applying the Gagliardo-Nirenberg inequality to $f(t,\cdot)$ from Lemma \ref{Lemma dis GN Hn}, we have 
\begin{align*}
\|f(t,\cdot)\|_{L^q(\mathbf{H}_n)}\lesssim \|f(t,\cdot)\|_{L^2(\mathbf{H}_n)}^{1-\theta(q)}\|\nabla_{\hor} f(t,\cdot)\|_{L^2(\mathbf{H}_n)}^{\theta(q)},
\end{align*} where $\theta(q)=\mathcal{Q}\big(\frac{1}{2}-\frac{1}{q}\big)$. 
Also, combining \eqref{2 lemma GN with weight 2 est} and \eqref{2 lemma GN with weight 3 est} with the last interpolative inequality, we obtain
\begin{align*}
\|f(t,\cdot)\|_{L^q(\mathbf{H}_n)} & \lesssim (1+t)^{(1-\theta(q))/2}\|\mathrm{e}^{\sigma\psi(t,\cdot)}\nabla_{\hor} v\|_{L^2(\mathbf{H}_n)} \\ &  \leq (1+t)^{(1-\theta(q))/2}\|\mathrm{e}^{\psi(t,\cdot)}\nabla_{\hor} v\|^{\sigma}_{L^2(\mathbf{H}_n)}\|\nabla_{\hor} v\|^{1-\sigma}_{L^2(\mathbf{H}_n)},
\end{align*} where in the last step we applied \eqref{2 lemma GN with weight 1 est}.
\end{proof}

\section{Local existence: proof of Theorem \ref{thm local existence}}\label{LE Section}

In the proof of Theorem \ref{thm local existence}, we employ the next result, which is a generalization to the non-linear case of Gronwall's lemma (cf. \cite[Section 3]{Bih56}).
\begin{lemma}[Bihari's inequality]\label{lemma of gronwall type} Let $k$ be a nonnegative, continuous function, $M$ a real constant and $g$ a continuous, non-decreasing, nonnegative function such that $$G(u)=\int_0^u \frac{ds}{g(s)}$$ is well-defined.
Let $y$ be a continuous function such that 
\begin{equation*}y(t)\leq M+\int_0^t k(s) g(y(s))ds 
\end{equation*} for any $t\geq 0$. Then, 
\begin{equation*}
G(y(t))\leq G(M)+\int_0^t k(s)ds 
\end{equation*} for any $t\geq 0$.
\end{lemma}

Using a standard contraction argument we prove now Theorem \ref{thm local existence}, following the main ideas of \cite[Appendix A]{IT05}. Note that differently from the global existence result, in this case we do not have to require a lower bound for the exponent $p$. 

\begin{proof}[Proof of Theorem \ref{thm local existence}]
Let $T,K$ be positive constants on which  will be prescribed several conditions of suitability throughout this proof. We define 
\begin{align*}
B_{T,K}^\psi \doteq \{v\in \mathcal{C}([0,T],H^1(\mathbf{H}_n))\cap\mathcal{C}^1([0,T],L^2(\mathbf{H}_n)):\|v\|_T^\psi\leq K\},
\end{align*} where the norm $\|\cdot\|_{T}^\psi$ is defined by
\begin{align*}
\|v\|_T^\psi \doteq \sup_{t\in[0,T]}\left(\|\mathrm{e}^{\psi(t,\cdot) }v(t,\cdot)\|_{L^2(\mathbf{H}_n)}+\|\mathrm{e}^{\psi(t,\cdot)}\nabla_{\hor} v(t,\cdot)\|_{L^2(\mathbf{H}_n)}+\|\mathrm{e}^{\psi(t,\cdot) }v_t(t,\cdot)\|_{L^2(\mathbf{H}_n)}\right).
\end{align*}
We introduce the map 
\begin{align*}
\Phi: B_{T,K}^\psi&\longrightarrow \mathcal{C}([0,T],H^1(\mathbf{H}_n))\cap\mathcal{C}^1([0,T],L^2(\mathbf{H}_n)), \\  v&\longmapsto u=\Phi(v),
\end{align*} where $u$ solves the Cauchy problem
\begin{align*}
\begin{cases}
u_{tt} -\Delta_{\hor} u + u_t =|v|^p, & (t,\eta)\in (0,T)\times\mathbf{H}_n, \\  u(0,\eta)=u_0(\eta), & \eta\in \mathbf{H}_n,\\  u_t(0,\eta)=u_1(\eta), & \eta\in \mathbf{H}_n.
\end{cases}
\end{align*} We shall prove that, for a suitable choice of $T$ and $K$, $\Phi$ is a contraction map from $B_{T,K}^\psi$ to itself.
From \eqref{fundamental equality 1} it results
\begin{align*}
\mathrm{e}^{2\psi}u_t|v|^p\geq \frac{\partial}{\partial t}\bigg(\frac{\mathrm{e}^{2\psi}}{2}(u_t^2+|\nabla_{\hor} u|^2)\bigg)-\diver(\mathrm{e}^{2\psi}u_t\nabla_{\hor} u).
\end{align*}
So, introducing the \emph{weighted energy} of the function $u$
\begin{align}
\mathcal{E}_{\psi}[u](t)\doteq \frac{1}{2}\int_{\mathbf{H}_n}\mathrm{e}^{2\psi(t,\eta)}\Big(|u_t(t,\eta)|^2+|\nabla_{\hor} u(t,\eta)|^2\Big) \mathrm{d} \eta \label{def weighted energy}
\end{align}  and integrating over $[0,t]\times\mathbf{H}_n$ the last inequality, we have
\begin{align*}
\mathcal{E}_{\psi}[u](t)\leq\mathcal{E}_{\psi}[u](0)+\int_0^t\int_{\mathbf{H}_n}\mathrm{e}^{2\psi(s,\eta)}u_t(s,\eta)|v(s,\eta)|^p \,\mathrm{d}\eta \, \mathrm{d}s,
\end{align*} where we used the divergence theorem.
Applying Cauchy-Schwarz inequality, we obtain
\begin{align*}
\mathcal{E}_{\psi}[u](t)&\leq \mathcal{E}_{\psi}[u](0)+\int_0^t\left(\int_{\mathbf{H}_n}\mathrm{e}^{2\psi(s,\eta)}|v(s,\eta)|^{2p} \mathrm{d}\eta\right)^{\frac{1}{2}} \left(\int_{\mathbf{H}_n}\mathrm{e}^{2\psi(s,\eta)}|u_t(s,\eta)|^2 \mathrm{d}\eta\right)^{\frac{1}{2}} \mathrm{d}s\\
&\leq \mathcal{E}_{\psi}[u](0)+\sqrt{2}\int_0^t\left(\int_{\mathbf{H}_n}\mathrm{e}^{2\psi(s,\eta)}|v(s,\eta)|^{2p} \mathrm{d}\eta\right)^{\frac{1}{2}}\mathcal{E}_{\psi}[u](s)^{\frac{1}{2}} \mathrm{d}s.
\end{align*}
Thanks to Bihari's inequality, with $g(u)=(2u)^{\frac{1}{2}}$, we find
\begin{align}\label{Bihari est}
\mathcal{E}_{\psi}[u](t)^{\frac{1}{2}}\leq \mathcal{E}_{\psi}[u](0)^{\frac{1}{2}}+\frac{1}{\sqrt{2}}\int_0^t\left(\int_{\mathbf{H}_n}\mathrm{e}^{2\psi(s,\eta)}|v(s,\eta)|^{2p} \mathrm{d}\eta\right)^{\frac{1}{2}} \mathrm{d}s.
\end{align} 
The condition $v\in B_{T,K}^\psi$ implies $v(t,\cdot)\in H^1_{\psi(t,\cdot)}(\mathbf{H}_n)$ for any $t\in [0,T]$. Also, from Lemma \ref{2 lemma GN with weight} we get 
\begin{align*}
\int_{\mathbf{H}_n}\mathrm{e}^{2\psi(s,\eta)}|v(s,\eta)|^{2p} \mathrm{d} \eta &=\|\mathrm{e}^{\frac{1}{p}\psi(s,\cdot)}v(s,\cdot)\|^{2p}_{L^{2p}(\mathbf{H}_n)}\\ &\lesssim (1+s)^{p(1-\theta(2p))}\|\nabla_{\hor} v(s,\cdot)\|_{L^2(\mathbf{H}_n)}^{2(p-1)}\|\mathrm{e}^{\psi(s,\cdot)}\nabla_{\hor} v(s,\cdot)\|_{L^2(\mathbf{H}_n)}^{2}\\&\lesssim (1+s)^{p(1-\theta(2p))}\|\mathrm{e}^{\psi(s,\cdot)}\nabla_{\hor} v(s,\cdot)\|_{L^2(\mathbf{H}_n)}^{2p}\\ & \leq (1+s)^{p(1-\theta(2p))}K^{2p}.
\end{align*} 
Consequently, from \eqref{Bihari est} we have
\begin{align*}
\mathcal{E}_{\psi}[u](t)^{\frac{1}{2}}\leq \mathcal{E}_{\psi}[u](0)^{\frac{1}{2}}+C_p T(1+T)^{p(1-\theta(2p))/2}K^p,
\end{align*} where $C_p>0$ is a multiplicative  constant independent of $T$ and $K$ that may change from line to line  up to the end of the proof.
Therefore, we get
\begin{align}
\|\mathrm{e}^{\psi(t,\cdot)}u_t(t,\cdot)\|_{L^2(\mathbf{H}_n)}+\|\mathrm{e}^{\psi(t,\cdot)}\nabla_{\hor} u(t,\cdot)\|_{L^2(\mathbf{H}_n)}\leq C_p \mathcal{E}_{\psi}[u](0)^{\frac{1}{2}}+C_p T(1+T)^{p(1-\theta(2p))/2}K^p. \label{est A}
\end{align} 
On the other hand, since
\begin{align*}
\mathrm{e}^{\psi(t,\eta)} u(t,\eta) =\mathrm{e}^{\psi(t,\eta)} u_0(\eta) + \int_0^t  \mathrm{e}^{\psi(t,\eta)} u_t(s,\eta) \, \mathrm{d}s
\end{align*} and $\psi$ is decreasing with respect to $t$, we have
\begin{align*}
\|\mathrm{e}^{\psi(t,\cdot)}u(t,\cdot)\|_{L^2(\mathbf{H}_n)} & \leq \|\mathrm{e}^{\psi(t,\cdot)}u_0\|_{L^2(\mathbf{H}_n)}  +\int_0^t  \| \mathrm{e}^{\psi(t,\cdot)} u_t(s,\cdot)\|_{L^2(\mathbf{H}_n)}  \, \mathrm{d}s \\
& \leq \|\mathrm{e}^{\psi(t,\cdot)}u_0\|_{L^2(\mathbf{H}_n)}  +\int_0^t  \| \mathrm{e}^{\psi(s,\cdot)} u_t(s,\cdot)\|_{L^2(\mathbf{H}_n)}  \, \mathrm{d}s 
 \\
& \leq \|\mathrm{e}^{\psi(t,\cdot)}u_0\|_{L^2(\mathbf{H}_n)}  +C_p\mathcal{E}_{\psi}[u](0)^{\frac{1}{2}} T+C_pT^2(1+T)^{p(1-\theta(2p))/2}K^p,
\end{align*} where in the last step we used \eqref{est A}.
So, we have just proved that 
\begin{align*}
 \|\mathrm{e}^{\psi(t,\cdot)}u(t,\cdot)\|_{L^2(\mathbf{H}_n)} & + \|\mathrm{e}^{\psi(t,\cdot)}u_t(t,\cdot)\|_{L^2(\mathbf{H}_n)}+\|\mathrm{e}^{\psi(t,\cdot)}\nabla_{\hor} u(t,\cdot)\|_{L^2(\mathbf{H}_n)} \\ & \quad \leq \|\mathrm{e}^{\psi(t,\cdot)}u_0\|_{L^2(\mathbf{H}_n)}+ C_p(1+T) \mathcal{E}_{\psi}[u](0)^{\frac{1}{2}} +C_pT^2(1+T)^{p(1-\theta(2p))}K^p
\\ & \quad \leq \|\mathrm{e}^{\psi(t,\cdot)}u_0\|_{L^2(\mathbf{H}_n)}+ C_p(1+T) \left(  \|\mathrm{e}^{\psi(t,\cdot)}\nabla_{\hor} u_0\|_{L^2(\mathbf{H}_n)} +  \|\mathrm{e}^{\psi(t,\cdot)}u_1\|_{L^2(\mathbf{H}_n)}\right) \\ & \quad \qquad+C_pT^2(1+T)^{p(1-\theta(2p))/2}K^p.
\end{align*}
Clearly, we may take $K$ sufficiently large such that $$\tfrac{K}{2}>\|\mathrm{e}^{\psi(t,\cdot)}u_0\|_{L^2(\mathbf{H}_n)}+C_p \left(  \|\mathrm{e}^{\psi(t,\cdot)}\nabla_{\hor} u_0\|_{L^2(\mathbf{H}_n)}+  \|\mathrm{e}^{\psi(t,\cdot)}u_1\|_{L^2(\mathbf{H}_n)}\right).$$ Hence, fixing now $T>0$  small enough so that $$\tfrac{K}{2}T+C_pT^2(1+T)^{p(1-\theta(2p))/2}K^p< \tfrac{K}{2},$$ since the above estimates are uniform in $t$, it follows that $\|v\|_{T}^\psi\leq K$, that is, $\Phi$ maps $B_{T,K}^\psi$ to itself. 

Finally, we have to  prove that $\Phi$ is a contraction map, provided that $T$ is sufficiently small. Let us take $v,\bar{v}\in B_{T,K}^\psi$. If we denote $u\doteq\Phi(u), \bar{u}\doteq\Phi(\bar{v})$, then, $w=u-\bar{u}$ solves the Cauchy problem
\begin{equation*}\begin{cases}
 w_{tt} -\Delta_{\hor} w +  w_t = |v|^p-|\bar{v}|^p, & (t,\eta)\in (0,T)\times\mathbf{H}_n, \\  u(0,\eta)=u_t(0,\eta)=0, &  \eta\in \mathbf{H}_n.
\end{cases}
\end{equation*} 
Using again \eqref{fundamental equality 1} and the divergence theorem, after integrating over $[0,t]\times \mathbf{H}_n$,  we get  the inequality
\begin{align*}
\mathcal{E}_{\psi}[w](t)\leq \int_0^t\int_{\mathbf{H}_n}\mathrm{e}^{2\psi(s,\eta)}\Big(|v(s,\eta)|^p-|\bar{v}(s,\eta)|^p\Big)w_t(s,\eta)\, \mathrm{d}\eta \, \mathrm{d}s.
\end{align*} By $||v|^p-|\bar{v}|^p|\leq p|v-\bar{v}|(|v|+|\bar{v}|)^{p-1}$ and Cauchy-Schwarz inequality, we arrive at
\begin{align*}
& \mathcal{E}_{\psi}[w](t)\lesssim  \int_0^t\int_{\mathbf{H}_n}\!\mathrm{e}^{2\psi(s,\eta)}|v(s,\eta)-\bar{v}(s,\eta)|\Big(|v(s,\eta)|+|\bar{v}(s,\eta)|\Big)^{p-1}w_t(s,\eta)\, \mathrm{d}\eta \, \mathrm{d}s \\
& \qquad \leq \int_0^t \left(\int_{\mathbf{H}_n}\!\mathrm{e}^{2\psi(s,\eta)}|w_t(s,\eta)|^2 \mathrm{d}\eta\right)^{\frac{1}{2}}\left(\int_{\mathbf{H}_n}\mathrm{e}^{2\psi(s,\eta)}|v(s,\eta)-\bar{v}(s,\eta)|^2\Big(|v(s,\eta)|+|\bar{v}(s,\eta)|\Big)^{2(p-1)} \mathrm{d}\eta\right)^{\frac{1}{2}} \mathrm{d}s
\\
& \qquad \leq \int_0^t\mathcal{E}_{\psi}[w](s)^{\frac{1}{2}}\left(\int_{\mathbf{H}_n}\!\mathrm{e}^{2\psi(s,\eta)}|v(s,\eta)-\bar{v}(s,\eta)|^2\Big(|v(s,\eta)|+|\bar{v}(s,\eta)|\Big)^{2(p-1)} \mathrm{d}\eta\right)^{\frac{1}{2}} \mathrm{d}s.
\end{align*}
Applying again Lemma \ref{lemma of gronwall type}, we find the inequality
\begin{align}\label{stima intermedia lemma loc esistenza}
\mathcal{E}_{\psi}[w](t)^\frac{1}{2}\lesssim   \int_0^t\left(\int_{\mathbf{H}_n}\mathrm{e}^{2\psi(s,\eta)}|v(s,\eta)-\bar{v}(s,\eta)|^2(|v(s,\eta)|+|\bar{v}(s,\eta)|)^{2(p-1)} \mathrm{d}\eta\right)^{\frac{1}{2}} \mathrm{d}s.
\end{align}
By H\"{o}lder's inequality it follows
\begin{align*}
&\|\mathrm{e}^{\psi(s,\cdot)}|v(s,\cdot)-\bar{v}(s,\cdot)|(|v(s,\cdot)|+|\bar{v}(s,\cdot)|)^{p-1}\|_{L^2(\mathbf{H}_n)}\\&\quad \qquad \leq \|\mathrm{e}^{(2-p)\psi(s,\cdot)}|v(s,\cdot)-\bar{v}(s,\cdot)|\|_{L^{2p}(\mathbf{H}_n)}\|\mathrm{e}^{(p-1)\psi(s,\cdot)}(|v(s,\cdot)|+|\bar{v}(s,\cdot)|)^{p-1}\|_{L^{\frac{2p}{p-1}}(\mathbf{H}_n)}.
\end{align*}
We estimate separately the two norms on the right-hand side of the last inequality. Using Lemma \ref{2 lemma GN with weight} and the property $\psi\geq 0$, we get
\begin{align*}
\|\mathrm{e}^{(2-p)\psi(s,\cdot)}|v(s,\cdot)-\bar{v}(s,\cdot)|\|_{L^{2p}(\mathbf{H}_n)}\lesssim (1+s)^{(1-\theta(2p))/2}\|\mathrm{e}^{\psi(s,\cdot)}\nabla_{\hor}(v(s,\cdot)-\bar{v}(s,\cdot))\|_{L^{2}(\mathbf{H}_n)}
\end{align*}
and
\begin{align*}
\|\mathrm{e}^{(p-1)\psi(s,\cdot)} &(|v(s,\cdot)|+|\bar{v}(s,\cdot)|)^{p-1}\|_{L^{\frac{2p}{p-1}}(\mathbf{H}_n)}=\left(\int_{\mathbf{H}_n}\mathrm{e}^{2p\psi(s,\eta)}(|v(s,\eta)|+|\bar{v}(s,\eta)|)^{2p} \mathrm{d}\eta\right)^{\frac{p-1}{2p}}\\&\quad\lesssim \left(\|\mathrm{e}^{\psi(s,\cdot)}v(s,\cdot)\|_{L^{2p}(\mathbf{H}_n)}+\|\mathrm{e}^{\psi(s,\cdot)}\bar{v}(s,\cdot)\|_{L^{2p}(\mathbf{H}_n)}\right)^{p-1}\\ &\quad\lesssim (1+s)^{(1-\theta(2p))(p-1)/2}\left(\|\mathrm{e}^{\psi(s,\cdot)}\nabla_{\hor} v(s,\cdot)\|_{L^{2}(\mathbf{H}_n)}+\|\mathrm{e}^{\psi(s,\cdot)}\nabla_{\hor} \bar{v}(s,\cdot)\|_{L^{2}(\mathbf{H}_n)}\right)^{p-1}.
\end{align*}
By \eqref{stima intermedia lemma loc esistenza} we have
\begin{align}
\|\mathrm{e}^{\psi(t,\cdot)}  w_t(t,\cdot)\|_{L^2(\mathbf{H}_n)} &+\|\mathrm{e}^{\psi(t,\cdot)}\nabla_{\hor} w(t,\cdot)\|_{L^2(\mathbf{H}_n)} \notag\\ & \leq C_p  \int_0^t (1+s)^{p(1-\theta(2p))/2}\|\mathrm{e}^{\psi(s,\cdot)}\nabla_{\hor}(v(s,\cdot)-\bar{v}(s,\cdot))\|_{L^{2}(\mathbf{H}_n)}\notag \\& \qquad \qquad\times\left(\|\mathrm{e}^{\psi(s,\cdot)}\nabla_{\hor} v(s,\cdot)\|_{L^{2}(\mathbf{H}_n)}+\|\mathrm{e}^{\psi(s,\cdot)}\nabla_{\hor} \bar{v}(s,\cdot)\|_{L^{2}(\mathbf{H}_n)}\right)^{p-1} \mathrm{d}s \notag\\ &\leq C_p\int_0^t (1+s)^{p(1-\theta(2p))/2}\mathrm{d}s \, \|v-\bar{v}\|_T^\psi\left(\|v\|_T^\psi+\|\bar{v}\|_T^\psi\right)^{p-1}\notag \\& \leq C_p T(1+T)^{p(1-\theta(2p))/2}K^{p-1} \|v-\bar{v}\|_T^\psi. \label{est B}
\end{align} 
Furthermore, 
\begin{align*}
\mathrm{e}^{\psi(t,\eta)} w(t,\eta) = \int_0^t  \mathrm{e}^{\psi(t,\eta)} w_t(s,\eta) \, \mathrm{d}s
\end{align*} and the fact that $	\psi $ is decreasing with respect to $t$ imply
\begin{align}
\|\mathrm{e}^{\psi(t,\cdot)}w(t,\cdot)\|_{L^2(\mathbf{H}_n)} & \leq \int_0^t  \| \mathrm{e}^{\psi(t,\eta)} w_t(s,\eta)\|_{L^2(\mathbf{H}_n)}  \, \mathrm{d}s  \leq \int_0^t  \| \mathrm{e}^{\psi(s,\eta)} w_t(s,\eta)\|_{L^2(\mathbf{H}_n)}  \, \mathrm{d}s \notag\\
& \leq C_p T^2(1+T)^{p(1-\theta(2p))/2}K^{p-1} \|v-\bar{v}\|_T^\psi,\label{est C}
\end{align} where in the last inequality we applied \eqref{est B}.
 Summarizing, combining \eqref{est B} and \eqref{est C} we arrive at
\begin{align*}
\|\Phi(v)-\Phi(\bar{v})\|_T^\psi  = \|w\|_T^\psi \leq C_p T(1+T)^{1+p(1-\theta(2p))/2}K^{p-1} \|v-\bar{v}\|_T^\psi.
\end{align*} So, choosing $T>0$ sufficiently small we find that $\Phi$ is a contraction.

Therefore, our starting problem has a unique solution $u$ in $\mathcal{C}([0,T_\mathrm{max}),H^1(\mathbf{H}_n))\cap\mathcal{C}^1([0,T_\mathrm{max}),L^2(\mathbf{H}_n))$ with finite energy $\mathcal{E}_{\psi}[u](t)$ for any $t\in [0,T_\mathrm{max})$, due to Banach's fixed point theorem.
Moreover $T_\mathrm{max}<\infty$ implies the blow up of the energy for $T\to T_\mathrm{max}^-$. Otherwise, if it was not so, we would have a finite energy for $u$ in a left neighborhood of $T_\mathrm{max}$, and then repeating the same arguments when the initial conditions are taken for $t=0$, we could extend the solution, violating the maximality of $T_\mathrm{max}$.
\end{proof}

\section{Estimates for the linear problem} \label{Section Linear estimates}

In order to prove Theorem \ref{thm glob exi exp data}, we recall some decay estimates for the solution of the linear Cauchy problem \eqref{linear CP}.
In the next propositions we can relax the assumptions for the initial data, considering a less restrictive space than the weighted energy space $\mathcal{A}(\mathbf{H}_n)$. More precisely, we may assume just data in the classical energy spaces with additional $L^1(\mathbf{H}_n)$ regularity, namely,
\begin{align*}
(u_0,u_1)\in (H^1(\mathbf{H}_n)\cap L^1(\mathbf{H}_n))\times (L^2(\mathbf{H}_n)\cap L^1(\mathbf{H}_n)).
\end{align*} 
We set
\begin{align*}
\mathcal{D}^{\kappa}(\mathbf{H}_n) \doteq (H^\kappa(\mathbf{H}_n)\cap L^1(\mathbf{H}_n))\times (L^2(\mathbf{H}_n)\cap L^1(\mathbf{H}_n)) \qquad \mbox{for } \, \kappa\in \{0,1\}. 
\end{align*} 
Clearly, 
\begin{align} \label{embedding spaces data}
\mathcal{A}(\mathbf{H}_n) \hookrightarrow \mathcal{D}^{1}_1(\mathbf{H}_n)\hookrightarrow H^1(\mathbf{H}_n)\times L^2(\mathbf{H}_n).
\end{align}

\begin{proposition}\label{Prop Lin Estim} Let us assume $(u_0,u_1)\in  \mathcal{D}^{1}(\mathbf{H}_n)$. Let $u\in\mathcal{C}([0,\infty), H^1(\mathbf{H}_n))\cap \mathcal{C}^1([0,\infty), L^2(\mathbf{H}_n))$ solve the Cauchy problem \eqref{linear CP}. Then, the following decay estimates are satisfied
\begin{align} \label{estimate u}
\| u(t,\cdot)\|_{L^2(\mathbf{H_n})} & \leq C (1+t)^{-\frac{\mathcal{Q}}{4}} \| (u_0,u_1)\|_{\mathcal{D}^0(\mathbf{H}_n)}  \\ 
\label{estimate nabla hor u} \| \nabla_{\hor} u(t,\cdot)\|_{L^2(\mathbf{H_n})} & \leq C (1+t)^{-\frac{\mathcal{Q}}{4}-\frac{1}{2}} \| (u_0,u_1)\|_{\mathcal{D}^1(\mathbf{H}_n)}\\ 
\label{estimate partial t u} \| \partial_t u(t,\cdot)\|_{L^2(\mathbf{H_n})} & \leq C (1+t)^{-\frac{\mathcal{Q}}{4}-1} \| (u_0,u_1)\|_{\mathcal{D}^1(\mathbf{H}_n)}
\end{align}
for any $t\geq 0$. Furthermore, if  we assume just $(u_0,u_1)\in H^1(\mathbf{H}_n)\times L^2(\mathbf{H}_n)$, that is, we do not require additional $L^1(\mathbf{H}_n)$ regularity for the Cauchy data, then 
the following estimates are satisfied
\begin{align} \label{estimate u only L2}
\| u(t,\cdot)\|_{L^2(\mathbf{H_n})} & \leq C\| (u_0,u_1)\|_{ L^2(\mathbf{H}_n)}  \\ 
\label{estimate nabla hor u only L2} \| \nabla_{\hor} u(t,\cdot)\|_{L^2(\mathbf{H_n})} & \leq C (1+t)^{-\frac{1}{2}} \| (u_0,u_1)\|_{H^1(\mathbf{H}_n)\times L^2(\mathbf{H}_n)}\\ 
\label{estimate partial t u only L2} \| \partial_t u(t,\cdot)\|_{L^2(\mathbf{H_n})} & \leq C (1+t)^{-1} \| (u_0,u_1)\|_{H^1(\mathbf{H}_n)\times L^2(\mathbf{H}_n)}
\end{align}
for any $t\geq 0$. Here $C>0$ is a universal constant.
\end{proposition}
\begin{proof}
See \cite[Theorem 1.1]{Pal19}, where the group Fourier transform on $\mathbf{H}_n$ is applied to prove this result.
\end{proof}

Finally, let us point out explicitly that we can still employ the estimates derived in the previous proposition in order to estimate  Duhamel's integral term \eqref{Duhamel term}, as the operator $\partial_t^2-\Delta_{\hor}+\partial_t$ is invariant by time translations.

\section{Global existence of small data solutions: proof of Theorem \ref{thm glob exi exp data}}\label{SDGE Section}

In order to prove Theorem \ref{thm glob exi exp data}, first we have to prove the next preliminary lemma, which allows us to estimate the weighted energy \eqref{def weighted energy} of a local (in time) solution $u$ to \eqref{CP semilinear 1}.

\begin{lemma}\label{lemma before glob exist with exp weight} Let $n\geq 1$ and $p>1$ such that $p\leq \frac{\mathcal{Q}}{\mathcal{Q}-2}$. Let $(u_0,u_1)\in \mathcal{A}(\mathbf{H}_n)$. If $u$ solves
\begin{align*}
\begin{cases}u_{tt} -\Delta_{\hor} u + u_t =|u|^p, & (t,\eta)\in (0,T)\times\mathbf{H}_n, \\  u(0,\eta)=u_0(\eta), &  \eta\in \mathbf{H}_n,\\  u_t(0,\eta)=u_1(\eta), &  \eta\in \mathbf{H}_n,
\end{cases}
\end{align*} then, the following energy estimate holds for any $t\in [0,T)$ and for an arbitrary small $\delta>0$
\begin{equation}\label{lemma before glob exist with exp weight est fund}
\mathcal{E}_{\psi}[u](t)\lesssim I^2_{0}+I^{p+1}_{0}+\bigg(\sup_{s\in[0,t]}(1+s)^\delta\|\mathrm{e}^{\left(\frac{2}{p+1}+\delta\right)\psi(s,\cdot)}u(s,\cdot)\|_{L^{p+1}(\mathbf{H}_n)}\bigg)^{p+1},
\end{equation} where 
\begin{align*}
I^2_{0}& \doteq \int_{\mathbf{H}_n}\mathrm{e}^{2\psi(0,\eta)}\Big(|u_1(\eta)|^2+|\nabla_{\hor} u_0(\eta)|^2\Big)\, \mathrm{d}\eta.
\end{align*}
\end{lemma}

\begin{proof}
First we  prove that 
\begin{align}\label{lemma before glob exist with exp weight est 1}
\mathcal{E}_{\psi}[u](t)\lesssim I^2_{0}+I^{p+1}_{0}+\|\mathrm{e}^{\frac{2}{p+1}\psi(t,\cdot)}u(t,\cdot)\|_{L^{p+1}(\mathbf{H}_n)}^{p+1}+\int_0^t \int_{\mathbf{H}_n}|\psi_t(s,\eta)|\mathrm{e}^{2\psi(s,\eta)}|u(s,\eta)|^{p+1}\mathrm{d}\eta\, \mathrm{d}s.
\end{align}
Integrating the relation \eqref{fundamental equality 2} over $[0,t]\times\mathbf{H}_n$, we get immediately (after using the divergence theorem)
\begin{align*}
\mathcal{G}_{\psi}[u](t)\leq \mathcal{G}_{\psi}[u](0)-\tfrac{2}{p+1}\int_0^t\int_{\mathbf{H}_n}\psi_t(s,\eta)\mathrm{e}^{2\psi(s,\eta)}|u(s,\eta)|^{p}u(s,\eta)\, \mathrm{d}\eta \, \mathrm{d}s,
\end{align*} where
\begin{align*}
\mathcal{G}_{\psi}[u](t)\doteq \mathcal{E}_{\psi}[u](t)-\tfrac{1}{p+1}\int_{\mathbf{H}_n}\mathrm{e}^{2\psi(t,\eta)}|u(t,\eta)|^{p}u(t,\eta)\, \mathrm{d}\eta.
\end{align*}
 Consequently,
\begin{align*}
\mathcal{E}_{\psi}[u](t)&\leq \mathcal{G}_{\psi}[u](0)+\tfrac{1}{p+1}\int_{\mathbf{H}_n}\mathrm{e}^{2\psi(t,\eta)}|u(t,\eta)|^{p}u(t,\eta)\,\mathrm{d}\eta -\tfrac{2}{p+1}\int_0^t\int_{\mathbf{H}_n}\psi_t(s,\eta)\mathrm{e}^{2\psi(s,\eta)}|u(s,\eta)|^{p}u(s,\eta)\,\mathrm{d}\eta \,\mathrm{d}s\\
&\lesssim \mathcal{G}_{\psi}[u](0)+\|\mathrm{e}^{\frac{2}{p+1}\psi(t,\cdot)}u(t,\cdot)\|_{L^{p+1}(\mathbf{H}_n)}^{p+1}+\int_0^t\int_{\mathbf{H}_n}|\psi_t(s,\eta)|\mathrm{e}^{2\psi(s,\eta)}|u(s,\eta)|^{p+1}\,\mathrm{d}\eta \,\mathrm{d}s.
\end{align*}
 So, in order to prove \eqref{lemma before glob exist with exp weight est 1} we have just to show that $\mathcal{G}_{\psi}[u](0)\lesssim  I^2_0+I^{p+1}_0$. 
 Since 
\begin{align*}
\mathcal{G}_{\psi}[u](0)&=\mathcal{E}_{\psi}[u](0)-\tfrac{1}{p+1}\int_{\mathbf{H}_n}\mathrm{e}^{\psi(0,\eta)}|u_0(\eta)|^p u_0(\eta)\,\mathrm{d}\eta\lesssim I^2_0+\int_{\mathbf{H}_n}\mathrm{e}^{\psi(0,\eta)}|u_0(\eta)|^{p+1}\mathrm{d}\eta,
\end{align*} we have to prove only that $$\displaystyle{\int_{\mathbf{H}_n}\mathrm{e}^{\psi(0,\eta)}|u_0(\eta)|^{p+1}\mathrm{d}\eta\lesssim I^{p+1}_0}.$$
Because of $p+1<\frac{\mathcal{Q}}{\mathcal{Q}-2}+1<\frac{2\mathcal{Q}}{\mathcal{Q}-2}$, using the Sobolev embedding $$H^1(\mathbf{H}_n)\hookrightarrow L^{p+1}(\mathbf{H}_n)$$ which follows, for example,  from the special case  $\theta=1$ in Lemma \ref{Lemma dis GN Hn} by interpolation with the trivial embedding $H^1(\mathbf{H}_n)\hookrightarrow L^{2}(\mathbf{H}_n)$, we find
\begin{align*}
\int_{\mathbf{H}_n}\mathrm{e}^{\psi(0,\eta)}|u_0(\eta)|^{p+1}\mathrm{d}\eta & =\|\mathrm{e}^{\frac{1}{p+1}\psi(0,\cdot)}u_0\|_{L^{p+1}(\mathbf{H}_n)}^{p+1}\lesssim \|\mathrm{e}^{\frac{1}{p+1}\psi(0,\cdot)}u_0\|_{H^1(\mathbf{H}_n)}^{p+1}\\
&=\left(\int_{\mathbf{H}_n}\mathrm{e}^{\frac{2}{p+1}\psi(0,\eta)}\left(|u_0(\eta)|^{2}+|\nabla_{\hor} u_0(\eta)|^{2}+ (p+1)^{-2} |\nabla_{\hor}\psi(0,\eta)|^2|u_0(\eta)|^{2} \right)\mathrm{d}\eta\right)^{\frac{p+1}{2}}\\
&\lesssim \left(\int_{\mathbf{H}_n}\mathrm{e}^{\frac{2}{p+1}\psi(0,\eta)}\left(|u_0(\eta)|^{2}+|\nabla_{\hor} u_0(\eta)|^{2}+  \big(|x|^2\!+|y|^2\big) |u_0(\eta)|^{2} \right)\mathrm{d}\eta\right)^{\frac{p+1}{2}}\\
& \lesssim \left(\int_{\mathbf{H}_n}\mathrm{e}^{2\psi(0,\eta)}\left(|u_0(\eta)|^{2}+|\nabla_{\hor} u_0(\eta)|^{2}\right)\mathrm{d}\eta\right)^{\frac{p+1}{2}}\lesssim I_{0}^{p+1},
\end{align*} where $\eta=(x,y,\tau)$ and in the second last inequality we have used the fact that $p>1$ to get the estimate 
\begin{align*}
\left(1+\big(|x|^2\!+|y|^2\big)\right)\mathrm{e}^{\frac{2}{p+1}\psi(0,\eta)}\lesssim \mathrm{e}^{2\psi(0,\eta)}.
\end{align*} 
So, we proved \eqref{lemma before glob exist with exp weight est 1}.
 From the relation $\psi_t(s,\eta)=-(1+s)^{-1}\psi(s,\eta)$ it follows
\begin{align*}
|\psi_t(s,\eta)|\mathrm{e}^{(2-\gamma(p+1))\psi(s,\eta)}=\tfrac{1}{1+s}\psi(s,\eta)\mathrm{e}^{-\delta(p+1)\psi(s,\eta)}\lesssim (1+s)^{-1},
\end{align*} with $\gamma=\frac{2}{p+1}+\delta$ and $\delta>0$. Therefore,
\begin{align}
\int_0^t\int_{\mathbf{H}_n}|\psi_t(s,\eta)|\mathrm{e}^{2\psi(s,\eta)}|u(s,\eta)|^{p+1}\mathrm{d}x \, \mathrm{d}s & \lesssim \int_0^t (1+s)^{-1}\int_{\mathbf{H}_n}\mathrm{e}^{\gamma (p+1)\psi(s,\eta)}|u(s,\eta)|^{p+1}\mathrm{d}\eta \, \mathrm{d}s\notag \\
& \leq \sup_{s\in[0,t]}(1+s)^{\delta(p+1)}\|\mathrm{e}^{\gamma\psi(s,\cdot)}u(s,\cdot)\|_{L^{p+1}(\mathbf{H}_n)}^{p+1}\int_0^t (1+s)^{-1-\delta(p+1)}\mathrm{d}s\notag\\
& \lesssim \bigg(\sup_{s\in[0,t]}(1+s)^{\delta}\|\mathrm{e}^{\gamma\psi(s,\cdot)}u(s,\cdot)\|_{L^{p+1}(\mathbf{H}_n)}\bigg)^{p+1}.\label{lemma before glob exist with exp weight est 2}
\end{align} 
Finally, since $\gamma>\frac{2}{p+1}$ and $\delta>0$, we have trivially
\begin{align}\label{last estimate local existence}
\|\mathrm{e}^{\frac{2}{p+1}\psi(t,\cdot)}u(t,\cdot)\|_{L^{p+1}(\mathbf{H}_n)}^{p+1}\leq \left((1+t)^{\delta}\|\mathrm{e}^{\gamma\psi(t,\cdot)}u(t,\cdot)\|_{L^{p+1}(\mathbf{H}_n)}\right)^{p+1}.
\end{align} Hence, combining \eqref{last estimate local existence}, \eqref{lemma before glob exist with exp weight est 1} and \eqref{lemma before glob exist with exp weight est 2}, we get the desired estimate \eqref{lemma before glob exist with exp weight est fund}.
\end{proof} 

Combing the \emph{linear estimates} from Section \ref{Section Linear estimates} and Lemma \ref{lemma before glob exist with exp weight}, we can finally prove Theorem \ref{thm glob exi exp data}.

\begin{proof}[Proof of Theorem \ref{thm glob exi exp data}]
By contradiction, let us assume that for any $\varepsilon_0>0$ there exists data satisfying \eqref{global existence eponential weight data cond} such that the solution $u\in\mathcal{C}([0,T_\mathrm{max}),H^1_{\psi(t,\cdot)}(\mathbf{H}_n))\cap \mathcal{C}^1([0,T_\mathrm{max}),L^2_{\psi(t,\cdot)}(\mathbf{H}_n))$ to the corresponding problem, whose existence is guaranteed by Theorem \ref{thm local existence}, is not global in time, that means $T_\mathrm{max}<\infty$.

For any $T\in (0,T_\mathrm{max})$, we may define the Banach space 
\begin{align*}
X(T) \doteq \mathcal{C}([0,T],H^1_{\psi(t,\cdot)}(\mathbf{H}_n))\cap \mathcal{C}^1([0,T],L^2_{\psi(t,\cdot)}(\mathbf{H}_n)),
\end{align*} equipped with the norm 
\begin{align*}
\|u\|_{X(T)} \doteq \sup_{t\in[0,T]}\Big[&\| \mathrm{e}^{\psi(t,\cdot)}\nabla_{\hor} u(t,\cdot)\|_{L^2(\mathbf{H}_n)}+\|\mathrm{e}^{\psi(t,\cdot)}u_t(t,\cdot)\|_{L^2(\mathbf{H}_n)} +(1+t)^{\frac{\mathcal{Q}}{4}}\|u(t,\cdot)\|_{L^2(\mathbf{H}_n)}\\&\quad+(1+t)^{\frac{\mathcal{Q}}{4}+\frac{1}{2}}\|\nabla_{\hor} u(t,\cdot)\|_{L^2(\mathbf{H}_n)}+(1+t)^{\frac{\mathcal{Q}}{4}+1}\|u_t(t,\cdot)\|_{L^2(\mathbf{H}_n)}\Big)\Big].
\end{align*} 
By Lemma \ref{lemma before glob exist with exp weight} it follows that
\begin{align}
 \|\mathrm{e}^{\psi(t,\cdot)}u_t(t,\cdot)\|_{L^2(\mathbf{H}_n)} & +\|\mathrm{e}^{\psi(t,\cdot)}\nabla_{\hor} u(t,\cdot)\|_{L^2(\mathbf{H}_n)}\notag \\ &\lesssim   \varepsilon_0+\varepsilon_0^{\frac{p+1}{2}}+\bigg(\sup_{s\in[0,t]}(1+s)^\delta\|\mathrm{e}^{\left(\delta+\frac{2}{p+1}\right)\psi(s,\cdot)}u(s,\cdot)\|_{L^{p+1}(\mathbf{H}_n)}\bigg)^{\frac{p+1}{2}}.\label{global existence thm est 1}
\end{align} 
 As $2<p+1$ and $p+1<2p\leq \frac{2\mathcal{Q}}{\mathcal{Q}-2}$, we find that $\theta(p+1)\in (0,1]$. Besides, we may take $\delta>0$ sufficiently small such that $\delta+\frac{2}{p+1}<1$. Let us stress that throughout the proof we will prescribe further conditions that the quantity $\delta$ has to fulfill. Hence, by Lemma \ref{2 lemma GN with weight} we obtain
\begin{align*}
\|\mathrm{e}^{\left(\delta+\frac{2}{p+1}\right)\psi(s,\cdot)}u(s,\cdot)\|_{L^{p+1}(\mathbf{H}_n)}&\lesssim (1+s)^{\frac{1}{2}(1-\theta(p+1))}\|\nabla_{\hor} u(s,\cdot)\|_{L^2(\mathbf{H}_n)}^{1-\left(\delta+\frac{2}{p+1}\right)}\|\mathrm{e}^{\psi(s,\cdot)}\nabla_{\hor} u(s,\cdot)\|_{L^2(\mathbf{H}_n)}^{\delta+\frac{2}{p+1}}\\ 
&\lesssim (1+s)^{\frac{1}{2}(1-\theta(p+1))-\left(1-\left(\delta+\frac{2}{p+1}\right)\right)\left(\frac{\mathcal{Q}}{4}+\frac{1}{2}\right)}\|u\|_{X(t)}\\ 
&\lesssim (1+s)^{\frac{\mathcal{Q}+1}{p+1}-\frac{\mathcal{Q}}{2}+\delta\left(\frac{\mathcal{Q}}{4}+\frac{1}{2}\right)}\|u\|_{X(t)}
\end{align*} for any $s\in [0,t]$.
As we assume $p>p_{\Fuj}(\mathcal{Q})$ (which is equivalent to require that $\frac{\mathcal{Q}+1}{p+1}-\frac{\mathcal{Q}}{2}<0$), we may consider $\delta>0$ such that 
\begin{align*}
\tfrac{\mathcal{Q}+1}{p+1}-\tfrac{\mathcal{Q}}{2}+\delta\left(\tfrac{\mathcal{Q}}{4}+\tfrac{1}{2}+1\right)<0.
\end{align*} Therefore, by \eqref{global existence thm est 1} we have
\begin{align}
\|\mathrm{e}^{\psi(t,\cdot)}u_t(t,\cdot)\|_{L^2(\mathbf{H}_n)}  +\|\mathrm{e}^{\psi(t,\cdot)}\nabla_{\hor} u(t,\cdot)\|_{L^2(\mathbf{H}_n)}\lesssim   \varepsilon_0+\varepsilon_0^{\frac{p+1}{2}}+ \|u\|_{X(t)}^{\frac{p+1}{2}}. \label{control weighted norms}
\end{align}
Let us proceed now with the estimate of the not-weighted $L^2(\mathbf{H_n})$ - norms. We will follow precisely the computations for the Euclidean case (cf. \cite[Section 18.1]{ER18}).
Thus,
\begin{align}
\|\partial_t^\ell \nabla_{\hor}^k u(t,\cdot)\|_{L^2(\mathbf{H}_n)} &  \lesssim \varepsilon_0 \, (1+t)^{-\frac{\mathcal{Q}}{4}-\frac{k}{2}-\ell} +\int_0^{t/2} (1+t-s)^{-\frac{\mathcal{Q}}{4}-\frac{k}{2}-\ell} \Big(\|u(s,\cdot)\|^p_{L^p(\mathbf{H}_n)}+\|u(s,\cdot)\|^p_{L^{2p}(\mathbf{H}_n)}\Big)\mathrm{d}s \notag\\ & \quad +\int_{t/2}^t (1+t-s)^{-\frac{k}{2}-\ell}\|u(s,\cdot)\|^p_{L^{2p}(\mathbf{H}_n)}\, \mathrm{d}s\label{global existence thm est 1,5}
\end{align} for $k+\ell =0,1$, where we used \eqref{embedding spaces data} to estimate the solution of the corresponding linear homogeneous problem, the $L^1\cap L^2$ - $L^2$ estimates \eqref{estimate u}, \eqref{estimate nabla hor u} and \eqref{estimate partial t u} to estimate Duhamel's term on the interval $[0,t/2]$ and the $L^2$ - $L^2$ estimates \eqref{estimate u only L2}, \eqref{estimate nabla hor u only L2} and \eqref{estimate partial t u only L2} on the interval $[t/2,t]$.
 Applying \eqref{L1-L2 weight est} and \eqref{L2-L2 weight est} to $|u(s,\cdot)|^p$ with $\sigma=\delta p$ and using \eqref{GN inequality weighted} and the definition of the norm $\|\cdot\|_{X(t)}$, we arrive at
\begin{align*} 
\|u(s,\cdot)\|^p_{L^{p}(\mathbf{H}_n)} & \lesssim (1+s)^\frac{\mathcal{Q}}{4}\|\mathrm{e}^{\delta\psi(s,\cdot)}u(s,\cdot)\|_{L^{2p}(\mathbf{H}_n)}^p  \\&  \lesssim (1+s)^{\frac{\mathcal{Q}}{4}+\frac{p}{2}(1-\theta(2p))}\|\nabla_{\hor} u(s,\cdot)\|_{L^{2}(\mathbf{H}_n)}^{(1-\delta)p} \, \|\mathrm{e}^{\psi(s,\cdot)}\nabla_{\hor}u(s,\cdot)\|_{L^{2}(\mathbf{H}_n)}^{\delta p}\\
&  \lesssim (1+s)^{\frac{\mathcal{Q}}{4}+\frac{p}{2}(1-\theta(2p))-(1-\delta)p\left(\frac{\mathcal{Q}}{4}+\frac{1}{2}\right)}\|u\|_{X(t)}^{p}=(1+s)^{-\frac{\mathcal{Q}p}{2}+\frac{\mathcal{Q}}{2}+\delta p\left(\frac{\mathcal{Q}}{4}+\frac{1}{2}\right)}\|u\|_{X(t)}^{p}
\end{align*} and 
\begin{align*}
\|u(s,\cdot)\|^p_{L^{2p}(\mathbf{H}_n)} & \lesssim \|\mathrm{e}^{\delta\psi(s,\cdot)}u(s,\cdot)\|_{L^{2p}(\mathbf{H}_n)}^p  \\&  \lesssim (1+s)^{\frac{p}{2}(1-\theta(2p))}\|\nabla_{\hor} u(s,\cdot)\|_{L^{2}(\mathbf{H}_n)}^{(1-\delta)p} \, \|\mathrm{e}^{\psi(s,\cdot)}\nabla_{\hor}u(s,\cdot)\|_{L^{2}(\mathbf{H}_n)}^{\delta p}\\
&  \lesssim (1+s)^{\frac{p}{2}(1-\theta(2p))-(1-\delta)p\left(\frac{\mathcal{Q}}{4}+\frac{1}{2}\right)}\|u\|_{X(t)}^{p}=(1+s)^{-\frac{\mathcal{Q}p}{2}+\frac{\mathcal{Q}}{4}+\delta p\left(\frac{\mathcal{Q}}{4}+\frac{1}{2}\right)}\|u\|_{X(t)}^{p},
\end{align*}  where we might apply \eqref{GN inequality weighted} thanks to the upper bound $p\leq p_{\mathrm{GN}}(\mathcal{Q})$ that guarantees $\theta(2p)\in (0,1]$. We estimate separately the two integrals on the right-hand side of \eqref{global existence thm est 1,5}.

 Let us begin with the integral over $[0,t/2]$:
\begin{align*}
\int_0^{t/2} (1+t-s & )^{-\frac{\mathcal{Q}}{4}-\frac{k}{2}-\ell} \Big(\|u(s,\cdot)\|^p_{L^p(\mathbf{H}_n)}+\|u(s,\cdot)\|^p_{L^{2p}(\mathbf{H}_n)}\Big)\mathrm{d}s \\ 
&  \lesssim  \int_0^{t/2} (1+t-s)^{-\frac{\mathcal{Q}}{4}-\frac{k}{2}-\ell} (1+s)^{-\frac{\mathcal{Q}p}{2}+\frac{\mathcal{Q}}{2}+\delta p\left(\frac{\mathcal{Q}}{4}+\frac{1}{2}\right)} \, \mathrm{d}s \, \|u\|_{X(t)}^{p} \\ 
&  \lesssim  (1+t)^{-\frac{\mathcal{Q}}{4}-\frac{k}{2}-\ell}  \int_0^{t/2}(1+s)^{-\frac{\mathcal{Q}p}{2}+\frac{\mathcal{Q}}{2}+\delta p\left(\frac{\mathcal{Q}}{4}+\frac{1}{2}\right)} \, \mathrm{d}s \, \|u\|_{X(t)}^{p}.
\end{align*}
Since $p>p_{\Fuj}(\mathcal{Q})$ and, equivalently,  $-\frac{\mathcal{Q}p}{2}+\frac{\mathcal{Q}}{2}<-1$, we can find $\delta>0$ such that 
\begin{align}
-\tfrac{\mathcal{Q}p}{2}+\tfrac{\mathcal{Q}}{2}+\delta p\left(\tfrac{\mathcal{Q}}{4}+\tfrac{1}{2}\right)<-1. \label{condition delta}
\end{align}
Consequently, 
\begin{align*}
\int_0^{t/2} (1+t-s )^{-\frac{\mathcal{Q}}{4}-\frac{k}{2}-\ell} \Big(\|u(s,\cdot)\|^p_{L^p(\mathbf{H}_n)}+\|u(s,\cdot)\|^p_{L^{2p}(\mathbf{H}_n)}\Big)\mathrm{d}s \lesssim  (1+t)^{-\frac{\mathcal{Q}}{4}-\frac{k}{2}-\ell}  \, \|u\|_{X(t)}^{p}.
\end{align*} Using again \eqref{condition delta}, for the integral over $[t/2,t]$ we obtain
\begin{align*}
\int_{t/2}^t (1+t-s)^{-\frac{k}{2}-\ell}\|u(s,\cdot)\|^p_{L^{2p}(\mathbf{H}_n)}\, \mathrm{d}s &  \lesssim  \int_{t/2}^t (1+t-s)^{-\frac{k}{2}-\ell}(1+s)^{-\frac{\mathcal{Q}p}{2}+\frac{\mathcal{Q}}{4}+\delta p\left(\frac{\mathcal{Q}}{4}+\frac{1}{2}\right)}\, \mathrm{d}s \, \|u\|_{X(t)}^{p} \\
&  \lesssim  (1+t)^{-\frac{\mathcal{Q}p}{2}+\frac{\mathcal{Q}}{4}+\delta p\left(\frac{\mathcal{Q}}{4}+\frac{1}{2}\right)} \int_{t/2}^t (1+t-s)^{-\frac{k}{2}-\ell}\, \mathrm{d}s \, \|u\|_{X(t)}^{p} \\
&  \lesssim  (1+t)^{-\frac{\mathcal{Q}p}{2}+\frac{\mathcal{Q}}{4}+\delta p\left(\frac{\mathcal{Q}}{4}+\frac{1}{2}\right)-\frac{k}{2}-\ell+1} \left(\log(1+t)\right)^{\ell} \|u\|_{X(t)}^{p} \\
&  \lesssim  (1+t)^{-\frac{\mathcal{Q}}{4}-\frac{k}{2}-\ell} \, \|u\|_{X(t)}^{p}.
\end{align*} Summarizing, from \eqref{global existence thm est 1,5} we derived
\begin{align}
(1+t)^{\frac{\mathcal{Q}}{4}+\frac{k}{2}+\ell} \,  \|\partial_t^\ell \nabla_{\hor}^k u(t,\cdot)\|_{L^2(\mathbf{H}_n)} &  \lesssim \varepsilon_0 + \|u\|_{X(t)}^{p}. \label{control unweighted norms}
\end{align} 
 Therefore, combining \eqref{control weighted norms} and \eqref{control unweighted norms}, it follows 
\begin{align}\label{est for the norm of u in X(T)}
\|u\|_{X(T)}\lesssim \varepsilon_0+\varepsilon_0^{\frac{p+1}{2}}+ \|u\|_{X(T)}^{\frac{p+1}{2}}+\|u\|_{X(T)}^p.
\end{align}  If $\varepsilon_0>0$ is small enough, then, from the last inequality we get that $\|u\|_{X(T)}$ is uniformly bounded, more precisely, 
\begin{align} \label{uniform boundedness of u in X(T)}
\|u\|_{X(T)}\lesssim \varepsilon_0
\end{align} for any $T\in(0,T_\mathrm{max})$ (cf. \cite[Section 6]{Pal17}, for example). 
Besides, from
\begin{align*}
\mathrm{e}^{\psi(t,\eta)} u(t,\eta) = \mathrm{e}^{\psi(t,\eta)} u_0(\eta)+\int_0^t  \mathrm{e}^{\psi(t,\eta)} u_t(s,\eta) \, \mathrm{d}s
\end{align*} and by using the monotonicity of $\psi$ with respect to $t$, we get
\begin{align*}
\| \mathrm{e}^{\psi(t,\cdot)} u(t,\cdot)\|_{L^2(\mathbf{H}_n)} & \lesssim  \varepsilon_0 +\int_0^t \|  \mathrm{e}^{\psi(t,\cdot)} u_t(s,\cdot)\|_{L^2(\mathbf{H}_n)} \, \mathrm{d}s  \lesssim  \varepsilon_0 +\int_0^t \|  \mathrm{e}^{\psi(s,\cdot)} u_t(s,\cdot)\|_{L^2(\mathbf{H}_n)} \, \mathrm{d}s \\ & \lesssim \varepsilon_0 (1+T),
\end{align*} where in the last estimate we used \eqref{uniform boundedness of u in X(T)}.
 Therefore, if $T_\mathrm{max}<\infty$, then, it holds
\begin{align*}
\limsup_{T\to T_\mathrm{max}^-}\left(\| \mathrm{e}^{\psi(t,\cdot)} u(t,\cdot)\|_{L^2(\mathbf{H}_n)}+\| \mathrm{e}^{\psi(t,\cdot)}\nabla_{\hor} u(t,\cdot)\|_{L^2(\mathbf{H}_n)}\| \mathrm{e}^{\psi(t,\cdot)} u_t(t,\cdot)\|_{L^2(\mathbf{H}_n)}\right) \lesssim \varepsilon_0 (1+T) <\infty .
\end{align*} 
Nevertheless, this is impossible according to the last part of Theorem \ref{thm local existence}, so  $T_\mathrm{max}=\infty$, that is $u$, has to be a global solution. The decay estimates for $u$ and its first order derivatives from the statement follows by the relation \eqref{uniform boundedness of u in X(T)} which holds uniformly with respect to $T$. 
\end{proof}

\section{Blow-up: proof of Theorem \ref{thm blow up}} \label{Section Blow-up}

Before proving Theorem \ref{thm blow up}, we recall briefly the definition of weak solution to \eqref{CP semilinear 1}.

\begin{definition} A \emph{weak solution} of the Cauchy problem \eqref{CP semilinear 1} in $[0,T)\times \mathbf{H}_n$ is a function $u\in L^p_{\loc}([0,T)\times \mathbf{H}_n)$ that satisfies
\begin{align}
& \int_0^T \int_{\mathbf{H}_n}|u(t,\eta)|^p\varphi(t,\eta) \, \mathrm{d}\eta \, \mathrm{d}t+  \int_{\mathbf{H}_n} \big(u_0(\eta) +u_1(\eta) \big)\varphi(0,\eta) \, \mathrm{d}\eta -  \int_{\mathbf{H}_n} u_0(\eta) \partial_t\varphi(0,\eta) \, \mathrm{d}\eta \notag \\ & \qquad  = \int_0^T \int_{\mathbf{H}_n}u(t,\eta)\left( \partial_t^2\varphi(t,\eta) -\Delta_{\hor}  \varphi(t,\eta) - \partial_t\varphi(t,\eta) \right) \mathrm{d}\eta \, \mathrm{d}t \label{def weak sol formula}
\end{align}
for any $\varphi \in \mathcal{C}_0^\infty([0,T)\times \mathbf{H}_n)$. If $T=\infty$, we call $u$ a \emph{global} in time weak solution to \eqref{CP semilinear 1}, else we call $u$ a \emph{local} in time weak solution.
\end{definition}

\begin{proof}[Proof of Theorem \ref{thm blow up}]
We apply the so-called \emph{test function method}. By contradiction, we assume that there exists a global in time weak solution $u$ to \eqref{CP semilinear 1}.

 Let us consider two bump functions $\alpha\in \mathcal{C}_0^\infty(\mathbb{R}^n)$ and $\beta\in \mathcal{C}_0^\infty(\mathbb{R})$. Furthermore, we require that $\alpha,\beta$ are radial symmetric and decreasing with respect to the radial variable,  $\alpha=1$ on $B_{n}(\frac{1}{2})$, $\beta=1$ on $[-\frac{1}{4},\frac{1}{4}]$, $\supp\alpha \subset B_{n}(1)$ and $\supp\beta \subset (-1,1)$.
  If $R>1$ is a parameter, then, we define the test function $\varphi_R\in \mathcal{C}^\infty_0([0,\infty)\times \mathbb{R}^{2n+1})$ with separate variables as follows:
 \begin{align}
 \varphi_R(t,x,y,\tau) \doteq \beta\left(\tfrac{t}{R^2}\right)\alpha\left(\tfrac{x}{R}\right)\alpha\left(\tfrac{y}{R}\right)\beta\left(\tfrac{\tau}{R^2}\right) \quad \mbox{for any} \ (t,x,y,\tau)\in [0,\infty)\times \mathbb{R}^{2n+1}. \label{def varphiR}
 \end{align} It is well-know that
 \begin{align*}
 |\partial_j \alpha| & \lesssim \alpha^{\frac{1}{p}} \quad \mbox{for any} \ 1\leq j\leq n, \quad
  |\partial_j \partial_k \alpha|  \lesssim \alpha^{\frac{1}{p}} \quad \mbox{for any} \ 1\leq j,k \leq n, \quad
  |\beta '|  \lesssim \beta^{\frac{1}{p}}, \quad |\beta''| \lesssim \beta^{\frac{1}{p}}.
 \end{align*} Furthermore, $0\leq \alpha,\beta \leq 1$ implies immediately $\alpha\leq \alpha^{\frac{1}{p}} $ and $\beta\leq \beta^{\frac{1}{p}} $. Therefore, from the relations
\begin{align*}
 \partial_t \varphi_R(t,x,y,\tau) &= R^{-2} \beta'\left(\tfrac{t}{R^2}\right)\alpha\left(\tfrac{x}{R}\right)\alpha\left(\tfrac{y}{R}\right)\beta\left(\tfrac{\tau}{R^2}\right), \\
 \partial_t^2 \varphi_R(t,x,y,\tau) &= R^{-4} \beta''\left(\tfrac{t}{R^2}\right)\alpha\left(\tfrac{x}{R}\right)\alpha\left(\tfrac{y}{R}\right)\beta\left(\tfrac{\tau}{R^2}\right), \\
 \Delta_{\hor}\varphi_R(t,x,y,\tau) &= R^{-2} \beta\left(\tfrac{t}{R^2}\right)\Delta \alpha\left(\tfrac{x}{R}\right)\alpha\left(\tfrac{y}{R}\right)\beta\left(\tfrac{\tau}{R^2}\right) +R^{-2} \beta\left(\tfrac{t}{R^2}\right) \alpha\left(\tfrac{x}{R}\right)\Delta \alpha\left(\tfrac{y}{R}\right)\beta\left(\tfrac{\tau}{R^2}\right) \\
 & \quad +  R^{-3}\sum _{j=1}^n x_j \beta\left(\tfrac{t}{R^2}\right)\alpha\left(\tfrac{x}{R}\right)\partial_j \alpha\left(\tfrac{y}{R}\right)\beta'\left(\tfrac{\tau}{R^2}\right) -  R^{-3}\sum _{j=1}^n y_j \beta\left(\tfrac{t}{R^2}\right)\partial_j \alpha\left(\tfrac{x}{R}\right)\alpha\left(\tfrac{y}{R}\right)\beta'\left(\tfrac{\tau}{R^2}\right)  \\
 & \quad + \tfrac14 R^{-4} (|x|^2+|y|^2) \beta\left(\tfrac{t}{R^2}\right)\alpha\left(\tfrac{x}{R}\right)\alpha\left(\tfrac{y}{R}\right)\beta''\left(\tfrac{\tau}{R^2}\right),
\end{align*} where $\Delta$ denotes the Laplace operator on $\mathbb{R}^n$, we get
\begin{equation} \label{estimate derivatives of varphi_R}
\begin{split}
 | \partial_t \varphi_R| & \lesssim R^{-2} ( \varphi_R)^{\frac{1}{p}},  \\  | \partial_t^2 \varphi_R| & \lesssim R^{-4} ( \varphi_R)^{\frac{1}{p}} \lesssim R^{-2} ( \varphi_R)^{\frac{1}{p}},  \\ | \Delta_{\hor} \varphi_R| & \lesssim R^{-2} (\varphi_R)^{\frac{1}{p}}.
 \end{split}
\end{equation} We used that $\supp \varphi_R \subset [0,R^2]\times B^n(R)\times B^n(R) \times [-R^2,R^2]$ in order to estimate the polynomial terms in the estimate of $|\Delta_{\hor} \varphi_R|$.

Let us apply the definition of weak solution \eqref{def weak sol formula} to the test function $\varphi_R$. Hence, by \eqref{estimate derivatives of varphi_R} we obtain
\begin{align}
 \int_0^\infty \int_{\mathbf{H}_n} & |u(t,\eta)|^p\varphi_R(t,\eta) \, \mathrm{d}\eta \, dt+  \int_{\mathbf{H}_n}\big(u_0(\eta) +u_1(\eta)\big)\varphi_R(0,\eta) \, \mathrm{d}\eta \notag\\
 & \leq  \int_0^\infty \int_{\mathbf{H}_n}|u(t,\eta)|\big(|\partial_t^2 \varphi_R(t,\eta)| +|\Delta_{\hor} \varphi_R(t,\eta)|+|\partial_t \varphi_R(t,\eta)|\big)  \, \mathrm{d}\eta \, \mathrm{d}t \notag \\
 &  \lesssim  R^{-2} \int_0^\infty \int_{\mathbf{H}_n}|u(t,\eta)| (\varphi_R(t,\eta))^{\frac{1}{p}} \, \mathrm{d}\eta \, \mathrm{d}t \notag  \\
 & \leq  R^{-2}  \bigg(\int_0^\infty \int_{\mathbf{H}_n} |u(t,\eta)|^p \varphi_R(t,\eta) \, \mathrm{d}\eta \, \mathrm{d}t \bigg)^{\frac{1}{p}}\bigg(\iint_{[0,R^2]\times \mathcal{D}_R} \mathrm{d}\eta \, \mathrm{d}t  \bigg)^{\frac{1}{p'}},\label{chain ineq weak sol varphi_R}
\end{align} 
where  in the last step we used H\"older's inquality and the support property for $\varphi_R$.
Let us introduce now the $R$-dependent integrals
\begin{align}
I_{R}\doteq \int_0^\infty \int_{\mathbf{H}_n}  |u(t,\eta)|^p\varphi_R(t,\eta) \, \mathrm{d}\eta \, \mathrm{d}t, \quad
J_{R}\doteq \int_{\mathbf{H}_n}\big(u_0(\eta)+u_1(\eta)\big)\varphi_R(0,\eta) \, \mathrm{d}\eta. \label{def IR and JR}
\end{align} Due to the assumption on the data in \eqref{assumption intial data TFM}, we have $\liminf_{R\to \infty} J_R>0$, which implies in turn that $J_R>0$ for $R\geq R_0$, where $R_0$ is a suitable positive real number. Indeed, from $\supp \varphi_R(0,\cdot)\subset \mathcal{D}_R$ and $\varphi_R(0,\cdot)=1$ on $\mathcal{D}_{R/2}$ we get trivially
\begin{align*}
J_R= \int_{\mathcal{D}_R}\big(u_0(\eta)+u_1(\eta)\big)\varphi_R(0,\eta) \, \mathrm{d}\eta \geq \int_{\mathcal{D}_{R/2}}\big(u_0(\eta)+u_1(\eta)\big) \, \mathrm{d}\eta.
\end{align*} Then, for $R\geq R_0$ the estimate in \eqref{chain ineq weak sol varphi_R} yields
\begin{align}\label{estimate I_R intermediate}
I_R\leq I_R+ J_R \lesssim R^{-2+\frac{2n+4}{p'}} I_R^{\frac{1}{p}}=R^{\mathcal{Q}-\frac{\mathcal{Q}+2}{p}} I_R^{\frac{1}{p}},
\end{align} where we applied $\mathrm{meas}(\mathcal{D}_R)\approx R^{\mathcal{Q}}$. When the exponent of $R$ in the right-hand side of the last inequality  is negative, i.e., for $p<p_{\Fuj}(\mathcal{Q})$, we have that $$0\leq I_R^{1-\frac{1}{p}}\lesssim R^{\mathcal{Q}-\frac{\mathcal{Q}+2}{p}}\longrightarrow  0 \quad \mbox{as}\  R\to \infty.$$ Thus, $\lim_{R\to \infty} I_R =0$. However, this is not possible, because the term $J_R$ is positive for $R$ sufficiently large. So, letting $R\to \infty$ in \eqref{estimate I_R intermediate} we find the contradiction we were looking for. In order to get a contradiction in the critical case $p=p_{\Fuj}(\mathcal{Q})$ too, we need to refine the estimate in \eqref{chain ineq weak sol varphi_R}. Indeed, we can use the fact that $\partial_t \varphi_R$ is supported in $\widehat{\mathcal{P}}_R\doteq [\frac{R^2}{4},R^2]\times \mathcal{D}_R$ and $\Delta_{\hor} \varphi_R$  is supported in $\widetilde{\mathcal{P}}_R\doteq[0,R^2]\times (\mathcal{D}_{1,R}\cup \mathcal{D}_{2,R}\cup \mathcal{D}_{3,R})$, where
\begin{align*}
\mathcal{D}_{1,R} & \doteq \big(B_n(R)\setminus B_n(R/2)\big)\times B_n(R)\times [-R^2,R^2] ,\\
\mathcal{D}_{2,R} & \doteq B_n(R)\times \big(B_n(R)\setminus B_n(R/2)\big)\times  [-R^2,R^2] ,\\
\mathcal{D}_{3,R} & \doteq  B_n(R) \times (B_n(R))\times \big( [-R^2,R^2]\setminus [-R^2/4,R^2/4]\big).
\end{align*} Consequently, for $R\geq R_0$ we may improve \eqref{chain ineq weak sol varphi_R} as follows
\begin{align}\label{estimate I_R intermediate improvement}
I_R\leq I_R+ J_R\lesssim  \widehat{I}_R^{\frac{1}{p}}+\widetilde{I}_R^{\frac{1}{p}},
\end{align} where
\begin{align*}
\widehat{I}_R  \doteq \iint_{\widehat{\mathcal{P}}_R}  |u(t,\eta)|^p\varphi_R(t,\eta) \, \mathrm{d}\eta \, \mathrm{d}t \quad \mbox{and} \quad
\widetilde{I}_R \doteq \iint_{\widetilde{\mathcal{P}}_R}  |u(t,\eta)|^p\varphi_R(t,\eta) \, \mathrm{d}\eta \, \mathrm{d}t.
\end{align*} In the critical case $p=p_{\Fuj}(\mathcal{Q})$, from \eqref{estimate I_R intermediate} it follows that $I_R$ is uniformly bounded as $R\to\infty$. Using the monotone convergence theorem, we find
\begin{align*}
\lim_{R\to\infty} I_R =  \lim_{R\to\infty} \int_0^\infty \int_{\mathbf{H}_n}  |u(t,\eta)|^p\varphi_R(t,\eta) \, \mathrm{d}\eta \, \mathrm{d}t = \int_0^\infty \int_{\mathbf{H}_n}  |u(t,\eta)|^p  \, \mathrm{d}\eta \, \mathrm{d}t \lesssim 1.
\end{align*} This means that $u\in L^p([0,\infty)\times \mathbf{H}_n)$. Applying now the dominated convergence theorem, as the characteristic functions of the sets $\widehat{\mathcal{P}}_R$ and $\widetilde{\mathcal{P}}_R$ converge to the zero function for $R\to \infty$, we have
\begin{align*}
\lim_{R\to\infty} \widehat{I}_R &= \lim_{R\to\infty} \iint_{\widehat{\mathcal{P}}_R}  |u(t,\eta)|^p\varphi_R(t,\eta) \, \mathrm{d}\eta \, \mathrm{d}t =0,\\
\lim_{R\to\infty} \widetilde{I}_R &= \lim_{R\to\infty} \iint_{\widetilde{\mathcal{P}}_R}  |u(t,\eta)|^p\varphi_R(t,\eta) \, \mathrm{d}\eta \, \mathrm{d}t =0.
\end{align*} Also, letting $R\to \infty$, \eqref{estimate I_R intermediate improvement} implies $\lim_{R\to\infty}I_R=0$ which provides the desired contradiction in turn, as we have already seen in the subcritical case. The proof is completed.
\end{proof}



\section*{Acknowledgments}

V. Georgiev is supported in part by
GNAMPA - Gruppo Nazionale per l'Analisi Matematica,
la Probabilit\`a e le loro Applicazioni,
by Institute of Mathematics and Informatics,
Bulgarian Academy of Sciences, by Top Global University Project, Waseda University and  by the University of Pisa, Project PRA 2018 49.
A. Palmieri is supported by the University of Pisa, Project PRA 2018 49. 






\end{document}